\documentclass[a4paper,12pt,reqno]{amsart}

\usepackage[top=25truemm,bottom=25truemm,left=20truemm,right=20truemm]{geometry}

\usepackage{amsthm,amsfonts}
\usepackage{amssymb}
\usepackage{stmaryrd}
\usepackage{mathtools}
\usepackage{mathrsfs} 
\usepackage[bbgreekl]{mathbbol} 
\usepackage{tikz-cd}
\usepackage{todonotes} 
\usepackage{thmtools} 
\usepackage{ascmac}
\usepackage{enumitem}

\DeclareSymbolFontAlphabet{\mathbb}{AMSb}
\DeclareSymbolFontAlphabet{\mathbbl}{bbold}

\renewcommand{\epsilon}{\varepsilon}
\renewcommand{\phi}{\varphi}

\setlist[enumerate]{label*=(\roman*)}
\usepackage[backend=biber,
    style=alphabetic,
    sorting=nyt,
    maxnames=8,
    giveninits=true,
    backref=true,
    doi=false,
    isbn=false,
    eprint=false
]{biblatex}
\definecolor{darkblue}{rgb}{0.2,0,0.6}
\definecolor{darkgreen}{rgb}{0.2,0.5,0.2}
\usepackage[%
    colorlinks=true,%
    linkcolor=darkblue,%
    citecolor=green,%
    urlcolor=darkgreen,%
]{hyperref}
\usepackage{cleveref}

\crefname{equation}{}{}

\theoremstyle{definition}
\newtheorem{theorem}{Theorem}[section]
\crefname{theorem}{Theorem}{Theorems}
\newtheorem{definition}[theorem]{Definition}
\crefname{definition}{Definition}{Definitions}
\newtheorem{lemma}[theorem]{Lemma}
\crefname{lemma}{Lemma}{Lemmas}
\newtheorem{corollary}[theorem]{Corollary}
\crefname{corollary}{Corollary}{Corollaries}
\newtheorem{proposition}[theorem]{Proposition}
\crefname{proposition}{Proposition}{Propositions}
\newtheorem{remark}[theorem]{Remark}
\crefname{remark}{Remark}{Remarks}
\newtheorem{example}[theorem]{Example}
\crefname{example}{Example}{Examples}
\newtheorem{notation}[theorem]{Notation}
\crefname{notation}{Notation}{Notations}
\renewcommand{\-}{\mathchar`-}

\newcommand{\A}{\mathscr{A}}
\newcommand{\B}{\mathscr{B}}
\newcommand{\C}{\mathscr{C}}
\newcommand{\D}{\mathscr{D}}
\newcommand{\E}{\mathscr{E}}

\newcommand{\bA}{\mathbb{A}}
\newcommand{\bB}{\mathbb{B}}
\newcommand{\bC}{\mathbb{C}}
\newcommand{\bD}{\mathbb{D}}
\newcommand{\bE}{\mathbb{E}}
\newcommand{\bF}{\mathbb{F}}
\newcommand{\bI}{\mathbb{I}}
\newcommand{\bJ}{\mathbb{J}}
\newcommand{\bM}{\mathbb{M}}
\newcommand{\bN}{\mathbb{N}}
\newcommand{\bS}{\mathbb{S}}
\newcommand{\bT}{\mathbb{T}}
\newcommand{\bX}{\mathbb{X}}
\newcommand{\bZ}{\mathbb{Z}}

\newcommand{\Set}{\mathbf{Set}}
\newcommand{\Cat}{\mathbf{Cat}}
\newcommand{\CAT}{\mathbf{CAT}}
\newcommand{\Mon}{\mathbf{Mon}}
\newcommand{\Grp}{\mathbf{Grp}}

\newcommand{\Pos}{\mathbf{Pos}}

\newcommand{\1}{\mathbbl{1}}

\newcommand{\3}{\mathbbl{3}}

\newcommand{\op}{\mathrm{op}}

\newcommand{\ob}{\mathrm{ob}}
\newcommand{\mor}{\mathrm{mor}}

\newcommand{\dom}{\mathop{\mathrm{dom}}}
\newcommand{\cod}{\mathop{\mathrm{cod}}}

\newcommand{\arr}[1]{\overset{#1}{\rightarrow}}
\newcommand{\longarr}[1]{\mathrel{
\tikz\draw[->] (0,0) -- node[above=1.8pt,inner sep=0pt] {\small$#1$} (1,0);
}}
\newcommand{\pto}{\rightharpoonup} 

\newcommand{\const}[1]{\ulcorner #1\urcorner}

\newcommand{\Id}{\mathrm{Id}}
\newcommand{\id}{\mathrm{id}}

\newcommand{\Colim}[1]{\mathop{\underset{#1}{\mathrm{Colim}}}}

\newcommand{\incat}[1]{~~\text{in }{#1}}

\newcommand{\fp}[1]{{#1}_\mathrm{fp}}

\newcommand{\orth}[1]{{#1}^\bot} 
\newcommand{\rorth}{\mathbf{r}}
\newcommand{\lorth}{\mathbf{l}}

\newcommand{\tup}[1]{\vec{#1}}
\newcommand{\intpn}[2]{\mathord{\left\llbracket #1 \right\rrbracket_{#2}}} 
\newcommand{\defined}{\mathord{\downarrow}}
\newcommand{\seq}[1]{\mathrel{
\tikz\draw[|-] (0,0) -- node[above=1.8pt,inner sep=0pt] {\small$#1$} (1,0);
}}
\newcommand{\biseq}[1]{\mathrel{
\tikz\draw[|-|] (0,0) -- node[above=1.8pt,inner sep=0pt] {\small$#1$} (1,0);
}}
\newcommand{\PStr}{\mathrm{PStr}}
\newcommand{\PMod}{\mathrm{PMod}}
\newcommand{\Term}{\mathrm{Term}}
\newcommand{\repn}[1]{\left\langle #1 \right\rangle} 

\newcommand{\mon}{\mathrm{mon}}
\newcommand{\pos}{\mathrm{pos}}
\newcommand{\quiv}{\mathrm{quiv}}
\newcommand{\cat}{\mathrm{cat}}
\newcommand{\rsrel}{\mathrm{rsrel}}

\newcommand{\Mod}{\mathrm{Mod}}

\newcommand{\ar}{\mathrm{ar}}
\newcommand{\type}{\mathrm{type}}
\newcommand{\Alg}{\mathop{\mathit{Alg}}}

\newcommand{\pht}[2]{{\bT_{#1,#2}}} 
\newcommand{\mnd}[2]{{T_{#1,#2}}} 
\newcommand{\Th}{\mathbf{Th}}

\newcommand{\Mndf}{\mathbf{Mnd_f}}
\newcommand{\EM}{\mathrm{EM}}

\title{Birkhoff's variety theorem for relative algebraic theories}
\author{Yuto Kawase}
\address{Research Institute for Mathematical Sciences, Kyoto University, Kyoto 606-8502, Japan}
\email{ykawase@kurims.kyoto-u.ac.jp}
\date{\today}
\keywords{finitary monad, locally finitely presentable category, partial Horn theory, Birkhoff's variety theorem, HSP theorem, universal algebra, equational theory}
\thanks{The author is grateful to his supervisor Masahito Hasegawa for his support, and to Soichiro Fujii and Hisashi Aratake for their useful suggestions.}
\subjclass{18C10,18C15,18C35,18E45}

\begin{document}

\begin{abstract}
    An algebraic theory, sometimes called an equational theory, is a theory defined by finitary operations and equations, such as the theories of groups and of rings.
    It is well known that algebraic theories are equivalent to finitary monads on $\mathbf{Set}$.
    In this paper, we generalize this phenomenon to locally finitely presentable categories using partial Horn logic.
    For each locally finitely presentable category $\mathscr{A}$, we define an ``algebraic concept'' relative to $\mathscr{A}$, which will be called an $\mathscr{A}$-relative algebraic theory, and show that $\mathscr{A}$-relative algebraic theories are equivalent to finitary monads on $\mathscr{A}$.
    In establishing such equivalence, a generalized Birkhoff's variety theorem plays an important role.
\end{abstract}

\maketitle

\tableofcontents
\section{Introduction}

The study initiated by Birkhoff in \cite{birkhoff1935structure} is called \emph{universal algebra}: a unified investigation of algebraic theories such as the theories of groups, rings, and Lie algebras.
One category-theoretic approach to universal algebra was taken by Linton, who implicitly showed in his paper \cite{linton1969outline} that algebraic theories correspond to finitary monads on $\Set$, and this correspondence preserves the concept of models.
Such a correspondence between algebraic theories and monads can be immediately generalized to the multi-sorted case.
In fact, it is well-known that for a given set $S$, there is a correspondence between $S$-sorted algebraic theories and finitary monads on $\Set^S$.

It is natural to speculate that finitary monads over general categories would extend such classical algebraic theories and give ``algebraic theories'' over those categories rather than $\Set$.
As far as finitary monads on locally finitely presentable categories are concerned, Kelly and Power gave a nice presentation for such monads in \cite{kelly1993adjunctions}, which makes our speculation more plausible.
In fact, ``algebraic theories'' over locally finitely presentable categories and their correspondence with finitary monads are discussed by Rosick\'{y} \cite[Section 3]{rosicky2021metric}.
However, Rosick\'{y}'s approach is not so syntactic in the sense that neither logical terms nor logical formulas appear explicitly.
On the other hand, there are more syntactic approaches.
One of them is given by Ford, Milius, and Schr\"{o}der in \cite{ford2021monads}.
By giving ``algebraic theories'' over a category of models of a relational Horn theory (which is a special case of locally presentable categories) in a syntactic way, they describe finitary monads over that category.
Another one is given by Ad\'{a}mek et al. in \cite{adamek2021finitary}.
By giving so-called ordered algebraic theories in a syntactic way, they describe finitary monads on the category $\Pos$ of posets.

Our aim is by giving ``algebraic theories'' over locally finitely presentable categories in a syntactic way as in \cite{ford2021monads,adamek2021finitary}, to describe finitary monads on locally finitely presentable categories and to give a framework for universal algebra over locally finitely presentable categories rather than $\Set$.
In fact, we show that for a given locally finitely presentable category $\A$, finitary monads on $\A$ correspond to ``algebraic theories over $\A$'' and this correspondence preserves the concept of models.
This is one of our main results and will be shown in \cref{section5}.
We call the concept of ``algebraic theories over $\A$,'' which is equivalent to finitary monads on $\A$, \emph{$\A$-relative algebraic theories}.
Note that our results are in the $\Set$-enriched sense, whereas the results of Ford, Milius, Schr\"{o}der, and Ad\'{a}mek were in the enriched setting.

In a classical algebraic theory, each operator $\omega$ has a natural number $n$ called an \emph{arity}, where $\omega$ is an operator that takes $n$ elements $x_1,\dots,x_n$ as input and outputs one element $\omega(x_1,\dots,x_n)$.
In contrast, in our $\A$-relative algebraic theory, each operator $\omega$ has as its arity a logical formula $\phi$ written in a language intrinsic to $\A$, in addition to the natural number $n$.
Namely, $\omega$ is a partial operator whose input is a tuple of $n$ elements $x_1,\dots,x_n$ satisfying the condition $\phi(x_1,\dots,x_n)$.
For example, in the case of $\A:=\Pos$, the partially ordered set of natural numbers $\bN$ has the partial operation $-$ of subtraction, where $x-y$ is defined only when $y\le x$ holds.
In this case, $-$ can be thought of as having as its arity the logical formula $y\le x$ written in the language of $\Pos$.
However, such an explanation of $\A$-relative algebraic theories is somewhat less rigorous, so we need to use some logical theory to describe languages intrinsic to locally finitely presentable categories.
The following logical theories are all known to characterize locally finitely presentable categories:
\begin{itemize}
    \item Cartesian theories (see \cite[{}D1.3.4]{johnstone2002sketches}),
    \item Essentially algebraic theories (see \cite[3.34 Definition]{adamek1994locally}),
    \item Partial Horn theories (see \cite{palmgren2007partial}).
\end{itemize}
Perhaps any of these theories would be sufficient to define $\A$-relative algebraic theory rigorously.
However, we use partial Horn theories in this paper because partial operations appear centrally in $\A$-relative algebraic theory and because we want to treat relation symbols explicitly, as in $\Pos$.

In \cref{section2}, we recall the relationship between locally finitely presentable categories and partial Horn theories and establish the correspondence between finitely presentable objects and Horn formulas.
We will use this correspondence throughout this paper.
The construction here to obtain finitely presentable objects from Horn formulas is a generalization of the construction of the term model in \cite{palmgren2007partial}.
The corresponding result for Cartesian theories has already appeared in \cite{caramello2018theories}.
While \cref{prop:validity_for_PHL} shows that a category of models of a partial Horn theory is an orthogonality class (or an injectivity class), a special version of this assertion restricted to categories of structures has already appeared in \cite{bidlingmaier2018categories}.

In \cref{section3}, for a locally finitely presentable category $\A$, we define $\A$-relative algebraic theories using partial Horn theories and show one direction of the correspondence between finitary monads and relative algebraic theories, which is one of our main results (\cref{thm:equiv_between_monad_and_rat}).
The other direction is showed in \cref{section5}.

One of the famous results in classical $\Set$-based algebraic theories is the following Birkhoff's variety theorem in \cite{birkhoff1935structure} (also called HSP theorem):
\begin{theorem}[\cite{birkhoff1935structure}]\label{thm:birkhoff_classical}
    Let $(\Omega,E)$ be a single-sorted algebraic theory (in our language, $\Set$-relative algebraic theory).
    Then the following are equivalent for a full subcategory $\E\subseteq\Alg(\Omega,E)$.
    \begin{enumerate}
        \item $\E$ is an equational class, i.e., there is a set $E'$ of equations such that $\E=\Alg(\Omega,E+E')$.
        \item $\E\subseteq\Alg(\Omega,E)$ is closed under:
        \begin{itemize}
            \item products,
            \item subobjects,
            \item quotients, i.e., if $p:A\to B$ is a surjective morphism in $\Alg(\Omega,E)$ and $A$ belongs to $\E$, then $B$ also belongs to $\E$.
        \end{itemize}
    \end{enumerate}
\end{theorem}
In \cref{section4}, we prove a generalization of the above theorem to $\A$-relative algebraic theories (\cref{thm:birkhoff_for_rat}).
This generalization is another main result, which subsumes not only \cref{thm:birkhoff_classical} but also Birkhoff's theorem for quasi-varieties (see \cite[3.22 Theorem]{adamek1994locally}).
In this generalization, we replace subobjects with ``$\Sigma$-closed subobjects,'' quotients with ``$U$-retracts,'' and add the closedness condition under filtered colimits.
The closedness condition under filtered colimits cannot be removed even in the case of $\A:=\Set^S$, which is discussed in detail in \cite[10.23 Example]{adamek2010algebraic}.
The classical Birkhoff's theorem (\cref{thm:birkhoff_classical}) does not depend on syntax because surjections can be characterized by regular epimorphisms, and hence all the conditions of \cref{thm:birkhoff_classical} are purely category-theoretic.
In contrast, our theorem depends on syntax because the concept of ``$\Sigma$-closed subobjects'' depends on the choice of syntax for $\A$.
This is a notable feature of our generalization of Birkhoff's theorem.

In \cref{section5}, by using our generalized Birkhoff's theorem, we show the other direction of the correspondence between finitary monads and relative algebraic theories and complete the proof of such correspondence (\cref{thm:equiv_between_monad_and_rat}).
Moreover, we observe the relationship between the condition for monads of preserving sifted colimits and the totality of operators.
This observation is closely related to the study in \cite{rosicky2012strongly}.

\section{Preliminaries}\label{section2}
\subsection{Locally finitely presentable categories and partial Horn theories}

We recall the definition of a locally presentable category.

\begin{definition} Let $\lambda$ be an infinite regular cardinal.
    \begin{enumerate}
        \item
        A category $\C$ is \emph{$\lambda$-filtered} if every diagram $\bJ\to\C$ has a cocone where $\bJ$ has $<\lambda$ arrows.
        \item
        A \emph{$\lambda$-filtered colimit} is a colimit indexed by a small $\lambda$-filtered category.
        \item
        An object $A$ of a category $\A$ is \emph{$\lambda$-presentable} if its hom-functor
        \[
        \A(A,\bullet):\A\longrightarrow\Set
        \]
        preserves $\lambda$-filtered colimits.
        An $\aleph_0$-presentable object is also called \emph{finitely presentable}.
        \item
        A category $\A$ is \emph{locally $\lambda$-presentable} if it is locally small and cocomplete, and there exists a set $\mathcal{G}$ of $\lambda$-presentable objects such that every object is a $\lambda$-filtered colimit of objects from $\mathcal{G}$.
        A locally $\aleph_0$-presentable category is also called \emph{locally finitely presentable}.
        \item
        A category $\A$ is \emph{locally presentable} if it is locally $\kappa$-presentable for some $\kappa$.
    \end{enumerate}
\end{definition}

\begin{notation}
    Let us denote by $\fp{\A}$ the full subcategory of a category $\A$ consisting of all finitely presentable objects in $\A$.
\end{notation}

\begin{definition}
Let $\C$ be a locally small category.
    \begin{enumerate}
        \item
        A class $\mathcal{G}\subseteq\C$ of objects in $\C$ is a \emph{generator} for $\C$ if the functor
        \begin{equation}\label{eq:generator}
            \C\ni C\mapsto \{\C(G,C)\}_{G\in\mathcal{G}}\in\Set^\mathcal{G}
        \end{equation}
        is faithful.
        \item
        A class $\mathcal{G}\subseteq\C$ is a \emph{strong generator} if the functor \cref{eq:generator} is faithful and conservative, i.e., reflects isomorphisms.
    \end{enumerate}
\end{definition}

\begin{theorem}\label{thm:basic_properties_of_lfp}
    The following are elementary properties of locally presentable categories.
    \begin{enumerate}
        \item\label{thm:basic_properties_of_lfp-1}
        Let $\mathcal{G}$ be a strong generator for a locally finitely presentable category $\A$.
        If every object in $\mathcal{G}$ is finitely presentable in $\A$, then $\fp{\A}$ is the finite colimit closure of $\mathcal{G}$ in $\A$.
        \item
        In a locally presentable category, the class of $\lambda$-presentable objects is essentially small.
        In particular, $\fp{\A}$ is essentially small for every locally presentable category $\A$.
        \item
        Every locally presentable category is complete, wellpowered, and co-wellpowered.
    \end{enumerate}
\end{theorem}
\begin{proof}\quad
\begin{enumerate}
    \item
    Let $\bar{\mathcal{G}}$ be the finite colimit closure of $\mathcal{G}$ in $\A$.
    By \cite[1.11 Theorem]{adamek1994locally}, every object of $\A$ is a filtered colimit of objects of $\Bar{\mathcal{G}}$.
    Thus \cite[Remark 2.1.6]{makkai1989accessible} implies that $\fp{\A}$ is the retract closure of $\Bar{\mathcal{G}}$.
    Since retracts are finite colimits (absolute coequalizers), we conclude that $\fp{\A}$ is the finite colimit closure of $\mathcal{G}$.
    \item
    See \cite[Corollary 2.3.12']{makkai1989accessible}.
    \item
    See \cite[{}1.D]{adamek1994locally}.\qedhere
\end{enumerate}
\end{proof}

We now introduce partial Horn logic as presented in \cite{palmgren2007partial}.

\begin{definition}
    Let $S$ be a set.
    An \emph{$S$-sorted signature} $\Sigma$ consists of:
    \begin{itemize}
        \item a set $\Sigma_\mathrm{f}$ of function symbols,
        \item a set $\Sigma_\mathrm{r}$ of relation symbols
    \end{itemize}
    such that
    \begin{itemize}
        \item for each $f\in\Sigma_\mathrm{f}$, an arity $f:s_1\times\dots\times s_n\to s\,(n\in\bN, s_i,s\in S)$ is given;
        \item for each $R\in\Sigma_\mathrm{r}$, an arity $R:s_1\times\dots\times s_n\,(n\in\bN, s_i\in S)$ is given.
    \end{itemize}
\end{definition}

Given the set $S$ of sorts, we fix an $S$-sorted set $\mathrm{Var}=(\mathrm{Var}_s)_{s\in S}$ such that $\mathrm{Var}_s$ is countably infinite for each $s\in S$.
We assume $\mathrm{Var}_s\cap\mathrm{Var}_{s'}=\varnothing$ if $s\neq s'$.
An element $x\in\mathrm{Var}_s$ is called a \emph{variable of sort $s$}.
The notation $x{:}s$ means that $x$ is a variable of sort $s$.

\begin{definition}
    Let $\Sigma$ be an $S$-sorted signature.
    \begin{enumerate}
        \item
        \emph{Terms} (over $\Sigma$) and their types are defined by the following inductive rules:
        \begin{itemize}
            \item
            Given a variable $x$ of sort $s\in S$, $x$ is a term of type $s$;
            \item
            Given a function symbol $f\in\Sigma$ with arity $s_1\times\dots\times s_n\to s$ and terms $\tau_1,\dots,\tau_n$ where $\tau_i$ is of type $s_i$, 
            then $f(\tau_1,\dots,\tau_n)$ is a term of type $s$.
        \end{itemize}
        \item
        \emph{Horn formulas} (over $\Sigma$) are defined by the following inductive rules:
        \begin{itemize}
            \item
            Given a relation symbol $R\in\Sigma$ with arity $s_1\times\dots\times s_n$ and terms $\tau_1,\dots,\tau_n$ where $\tau_i$ is of type $s_i$, 
            then $R(\tau_1,\dots,\tau_n)$ is a Horn formula;
            \item
            Given two terms $\tau$ and $\tau'$ of the same type $s$, 
            then $\tau=\tau'$ is a Horn formula;
            \item
            The truth constant $\top$ is a Horn formula;
            \item
            Given two Horn formulas $\phi$ and $\psi$, $\phi\wedge\psi$ is a Horn formula.
        \end{itemize}
        \item
        A \emph{context} is a finite tuple $\tup{x}$ of distinct variables.
        \item
        A \emph{term-in-context} (over $\Sigma$) is a pair of a context $\tup{x}$ and term $\tau$ (over $\Sigma$), written as $\tup{x}.\tau$, where all variables appearing in $\tau$ occur in $\tup{x}$.
        \item
        A \emph{Horn formula-in-context} (over $\Sigma$) is a pair of a context $\tup{x}$ and Horn formula $\phi$ (over $\Sigma$), written as $\tup{x}.\phi$, where all variables appearing in $\phi$ occur in $\tup{x}$.
        \item
        A \emph{Horn sequent} (over $\Sigma$) is a pair of two Horn formulas $\phi$ and $\psi$ (over $\Sigma$) in the same context $\tup{x}$, written as
        \begin{equation*}
            \phi \seq{\tup{x}} \psi.
        \end{equation*}
        \item
        A \emph{partial Horn theory} $\bT$ (over $\Sigma$) is a set of Horn sequents (over $\Sigma$).
    \end{enumerate}
\end{definition}

Note that we do not consider equal sign $=$ to be a relation symbol, and in this paper, we do not use any term (resp. formula) in no context, so we call a term-in-context (resp. formula-in-context) simply a term (resp. formula).
We informally use the abbreviation $\phi\biseq{\tup{x}}\psi$ for ``$(\phi\seq{\tup{x}}\psi)$ and $(\psi\seq{\tup{x}}\phi)$,'' and $\tau\defined$ for $\tau=\tau$.

\begin{definition}
    Let $\Sigma$ be an $S$-sorted signature.
    A \emph{partial $\Sigma$-structure} $M$ consists of:
    \begin{itemize}
        \item
        a set $M_s$ for each sort $s\in S$,
        \item
        a partial map $\intpn{f}{M}$ (or $\intpn{\tup{x}.f(\tup{x})}{M}$) $:M_{s_1}\times\dots\times M_{s_n}\pto M_s$ for each function symbol $f:s_1\times\dots\times s_n\to s$ in $\Sigma$,
        \item
        a subset $\intpn{R}{M}$ (or $\intpn{\tup{x}.R(\tup{x})}{M}$) $\subseteq M_{s_1}\times\dots\times M_{s_n}$ for each relation symbol $R:s_1\times\dots\times s_n$ in $\Sigma$.
    \end{itemize}
\end{definition}

We can easily extend the above definitions of $\intpn{\tup{x}.f(\tup{x})}{M}$ and $\intpn{\tup{x}.R(\tup{x})}{M}$ to arbitrary terms-in-context and Horn formulas-in-context as follows:
\begin{itemize}
    \item
    $\intpn{\tup{y}.f(\tau_1,\dots,\tau_n)}{M}(\tup{m})$ is defined if and only if all $\intpn{\tup{y}.\tau_i}{M}(\tup{m})$ are defined and
    \[
    \intpn{f}{M}(\intpn{\tup{y}.\tau_1}{M}(\tup{m}),\dots,\intpn{\tup{y}.\tau_n}{M}(\tup{m}))
    \]
    is also defined, 
    and then
    \[
    \intpn{\tup{y}.f(\tau_1,\dots,\tau_n)}{M}(\tup{m}):=\intpn{f}{M}(\intpn{\tup{y}.\tau_1}{M}(\tup{m}),\dots,\intpn{\tup{y}.\tau_n}{M}(\tup{m}));
    \]
    \item
    $\tup{m}$ belongs to $\intpn{\tup{y}.R(\tau_1,\dots,\tau_n)}{M}$ if and only if all $\intpn{\tup{y}.\tau_i}{M}(\tup{m})$ are defined and
    \[
    (\intpn{\tup{y}.\tau_1}{M}(\tup{m}),\dots,\intpn{\tup{y}.\tau_n}{M}(\tup{m}))
    \]
    belongs to $\intpn{R}{M}$;
    \item
    $\tup{m}$ belongs to $\intpn{\tup{y}.\tau=\tau'}{M}$ if and only if both $\intpn{\tup{y}.\tau}{M}(\tup{m})$ and $\intpn{\tup{y}.\tau'}{M}(\tup{m})$ are defined and equal to each other;
    \item
    $\intpn{\tup{x}.\phi\wedge\psi}{M}:=\intpn{\tup{x}.\phi}{M}\cap\intpn{\tup{x}.\psi}{M}$;
    \item
    $\intpn{\tup{x}.\top}{M}:=\prod_i M_{s_i}$, where $x_i{:}s_i$.
\end{itemize}

\begin{definition}
    We say that a Horn sequent $\phi\seq{\tup{x}}\psi$ over $\Sigma$ is \emph{valid} in a partial $\Sigma$-structure $M$ if $\intpn{\tup{x}.\phi}{M}\subseteq\intpn{\tup{x}.\psi}{M}$.
    A partial $\Sigma$-structure $M$ is called a \emph{partial $\bT$-model} for a partial Horn theory $\bT$ over $\Sigma$ if all sequents in $\bT$ are valid in $M$.
\end{definition}

\begin{definition}
    Let $\bT$ be a partial Horn theory over $\Sigma$.
    A Horn sequent $\phi\seq{\tup{x}}\psi$ over $\Sigma$ is called a \emph{PHL-theorem} of $\bT$ if it is valid in every partial model of $\bT$.
\end{definition}

An alternative definition of PHL-theorem can also be given syntactically using axioms and inference rules \cite{palmgren2007partial}. 
The above definition is equivalent to this by the completeness theorem.

\begin{definition}
    Let $\Sigma$ be an $S$-sorted signature.
    A \emph{$\Sigma$-homomorphism} $h:M\to N$ between partial $\Sigma$-structures consists of:
    \begin{itemize}
        \item a total map $h_s:M_s\to N_s$ for each sort $s\in S$
    \end{itemize}
    such that for each function symbol $f:s_1\times\dots\times s_n\to s$ in $\Sigma$ and relation symbol $R:s_1\times\dots\times s_n$ in $\Sigma$, there exist (necessarily unique) total maps (denoted by dashed arrows) making the following diagrams commute:
    \[
    \begin{tikzcd}[column sep=large, row sep=large]
        M_{s_1}\times\dots\times M_{s_n}\arrow[d,"h_{s_1}\times\dots\times h_{s_n}"'] &[-10pt] \mathrm{Dom}(\intpn{f}{M})\arrow[d,"\exists"',dashed]\arrow[l,hook']\arrow[r,"\intpn{f}{M}"] &[20pt] M_s\arrow[d,"h_s"] \\
        N_{s_1}\times\dots\times N_{s_n} & \mathrm{Dom}(\intpn{f}{N})\arrow[l,hook']\arrow[r,"\intpn{f}{N}"'] & N_s
    \end{tikzcd}
    \]
    \[
    \begin{tikzcd}[column sep=large, row sep=large]
        M_{s_1}\times\dots\times M_{s_n}\arrow[d,"{h_{s_1}\times\dots\times h_{s_n}}"'] &[-10pt] \intpn{R}{M}\arrow[d,"\exists"',dashed]\arrow[l,hook'] \\
        N_{s_1}\times\dots\times N_{s_n} & {\intpn{R}{N}} \arrow[l,hook']
    \end{tikzcd}
    \]
\end{definition}

\begin{notation}
    Let $\bT$ be a partial Horn theory over an $S$-sorted signature $\Sigma$.
    We will denote by $\Sigma\-\PStr$ the category of partial $\Sigma$-structures and $\Sigma$-homomorphisms and by $\bT\-\PMod$ the full subcategory of $\Sigma\-\PStr$ consisting of all partial $\bT$-models.
\end{notation}

\begin{theorem}\label{thm:PHT_characterize_lfp}
    A category $\A$ is locally finitely presentable if and only if there exists a partial Horn theory $\bT$ such that $\A\simeq\bT\-\PMod$.
\end{theorem}
\begin{proof}
    See \cite[Theorem 9.6, Proposition 9.9]{palmgren2007partial} and \cite[{}1.46]{adamek1994locally}.
\end{proof}

\begin{example}[posets]
    We present the partial Horn theory $\bS_\pos$ for posets.
    Let $S:=\{ *\}$, $\Sigma_\pos:=\{ \le :*\times * \}$.
    The partial Horn theory $\bS_\pos$ over $\Sigma_\pos$ consists of:
    \begin{gather*}
        \top\seq{x}x\le x,\quad
        x\le y\wedge y\le x\seq{x,y}x=y,\quad
        x\le y\wedge y\le z\seq{x,y,z}x\le z.
    \end{gather*}
Then, we have $\bS_\pos\-\PMod\simeq\Pos$.
\end{example}

\begin{example}[small categories]\label{eg:PHT_for_smallcat}
    We present the partial Horn theory $\bS_\cat$ for small categories.
    Let $S:=\{ \ob, \mor\}$.
    The $S$-sorted signature $\Sigma_\cat$ consists of:
    \begin{gather*}
        \id:\ob\to\mor,\quad
        \mathrm{d}:\mor\to\ob,\quad
        \mathrm{c}:\mor\to\ob,\quad
        \circ:\mor\times\mor\to\mor.
    \end{gather*}
    The partial Horn theory $\bS_\cat$ over $\Sigma_\cat$ consists of:
    \begin{gather*}
        \top\seq{x{:}\ob}\id(x)\defined,\quad
        \top\seq{f{:}\mor}\mathrm{d}(f)\defined\wedge\mathrm{c}(f)\defined,\quad
        \mathrm{d}(g)=\mathrm{c}(f)\biseq{g,f{:}\mor}(g\circ f)\defined;\\
        \top\seq{x{:}\ob}\mathrm{d}(\id(x))=x\wedge\mathrm{c}(\id(x))=x;\\
        \mathrm{d}(g)=\mathrm{c}(f)\seq{g,f{:}\mor}\mathrm{d}(g\circ f)=\mathrm{d}(f)\wedge\mathrm{c}(g\circ f)=\mathrm{c}(g);\\
        \top\seq{f{:}\mor}f\circ\id(\mathrm{d}(f))=f\wedge \id(\mathrm{c}(f))\circ f=f;\\
        \mathrm{d}(h)=\mathrm{c}(g)\wedge\mathrm{d}(g)=\mathrm{c}(f)\seq{h,g,f{:}\mor}(h\circ g)\circ f=h\circ (g\circ f).
    \end{gather*}
Then, we have $\bS_\cat\-\PMod\simeq\Cat$.
\end{example}

We will denote by $(S,\Sigma,\bT)$ a partial Horn theory $\bT$ over an $S$-sorted signature $\Sigma$.

\begin{definition}
    Consider two partial Horn theories $(S,\Sigma,\bT)$ and $(S',\Sigma',\bT')$.
    A \emph{theory morphism}
    \[
    \rho:(S,\Sigma,\bT)\to(S',\Sigma',\bT')
    \]
    consists of:
    \begin{itemize}
        \item
        a map $S\ni s\mapsto s^\rho\in S'$;
        \item
        an assignment to each function symbol $f:s_1\times\dots\times s_n\to s$ in $\Sigma$, a function symbol $f^\rho:s_1^\rho\times\dots\times s_n^\rho\to s^\rho$ in $\Sigma'$;
        \item
        an assignment to each relation symbol $R:s_1\times\dots\times s_n$ in $\Sigma$, a relation symbol $R^\rho:s_1^\rho\times\dots\times s_n^\rho$ in $\Sigma'$
    \end{itemize}
    such that for every axiom $\phi\seq{\tup{x}}\psi$ in $\bT$, the \emph{$\rho$-translation} $\phi^\rho\seq{\tup{x}^\rho}\psi^\rho$ is a PHL-theorem of $\bT'$.
    Here $\tup{x}^\rho=(x_1^\rho{:}s_1^\rho,\dots,x_n^\rho{:}s_n^\rho)$ with $\tup{x}{:}\tup{s}$.
    The \emph{$\rho$-translation} $\phi^\rho$ and $\psi^\rho$ are constructed by replacing all symbols that $\phi$ and $\psi$ include by $\rho$.
    For more details we refer the reader to \cite{palmgren2007partial}.
\end{definition}

Our definition of theory morphisms seems strict since it completely depends on syntax.
A looser definition is introduced in \cite[Definition 9.10]{palmgren2007partial}.

A theory morphism $\rho:\bT\to\bT'$ induces the forgetful functor $U^\rho:\bT'\-\PMod\to\bT\-\PMod$.

\begin{theorem}
    For every theory morphism $\rho:\bT\to\bT'$, the forgetful functor $U^\rho$ has a left adjoint $F^\rho$.
\end{theorem}
\begin{proof}
    See \cite[Theorem 5.4]{palmgren2007partial}.
\end{proof}

\subsection{The correspondence between finitely presentable objects and Horn formulas}

\begin{definition}\label{def:T-term}
Let $\bT$ be a partial Horn theory over an $S$-sorted signature $\Sigma$.
Let $\tup{x}.\phi$ be a Horn formula-in-context over $\Sigma$, where $\tup{x}=(x_i{:}s_i)_{1\le i\le n}$.
\begin{enumerate}
    \item
    \emph{$\bT$-terms} (generated by $\tup{x}.\phi$) and their types are defined inductively as follows:
    \begin{itemize}
        \item
        $x_i~(1\le i\le n)$ is a $\bT$-term of type $s_i$;
        \item
        Given a function symbol $f:s'_1\times\dots\times s'_m\to s'$ and $\bT$-terms $\tau_j~(1{\le} j{\le} m)$ of type $s'_j$, 
        if $\phi\seq{\tup{x}}f(\tau_1,\dots,\tau_m)\defined$ is a PHL-theorem of $\bT$, then $f(\tau_1,\dots,\tau_m)$ is a $\bT$-term of type $s'$.
    \end{itemize}
    We write $\bT\-\Term(\tup{x}.\phi)$ for the $S$-sorted set of all $\bT$-terms generated by $\tup{x}.\phi$.
    \item
    Define a congruence $\approx_\bT$ on $\bT\-\Term(\tup{x}.\phi)$ as follows: 
    $\tau\approx_\bT\tau'$ if and only if $\phi\seq{\tup{x}}\tau =\tau'$ is a PHL-theorem of $\bT$.
    We will denote by $[\tau]_\bT$ the equivalence class of $\tau$ by $\approx_\bT$.
    \item
    The quotient $S$-sorted set $M:=\bT\-\Term(\tup{x}.\phi)/{\approx_\bT}$ becomes a partial $\Sigma$-structure as follows:
    \begin{itemize}
        \item
        For each function symbol $f\in\Sigma$, define the partial function
        \[
        \intpn{f}{M}:([\tau_1]_\bT,\dots,[\tau_m]_\bT)\mapsto [f(\tau_1,\dots,\tau_m)]_\bT
        \]
        if and only if $\phi\seq{\tup{x}}f(\tau_1,\dots,\tau_m)\defined$ is a PHL-theorem of $\bT$.
        \item
        For each relation symbol $R\in\Sigma$, 
        \begin{equation*}
            \intpn{R}{M}:=\{([\tau_1]_\bT,\dots,[\tau_m]_\bT)\mid\phi\seq{\tup{x}}R(\tau_1,\dots,\tau_m)\text{ is a PHL-theorem of }\bT\}.
        \end{equation*}
    \end{itemize}
\end{enumerate}
\end{definition}

\begin{example}
    Let $\bS_\cat$ be the partial Horn theory for small categories as in \cref{eg:PHT_for_smallcat}.
    Then the term $\mathrm{d}(g\circ f)$ over $\Sigma_\cat$ is an $\bS_\cat$-term generated by $(g,f{:}\mor).\mathrm{d}(g)=\mathrm{c}(f)$ but not by $(g,f{:}\mor).\top$.
    The $\Sigma_\cat$ structure $\bS_\cat\-\Term((g,f{:}\mor).\mathrm{d}(g)=\mathrm{c}(f))/{\approx_{\bS_\cat}}$ is the small category $\3$:
    \[
    \3:=\{ \cdot\arr{f} \cdot\arr{g} \cdot\}.
    \]
\end{example}

\begin{lemma}\label{lem:basic_property_repn_model}
    Let $\bT$ be a partial Horn theory over an $S$-sorted signature $\Sigma$.
    Let $\tup{x}.\phi$ be a Horn formula-in-context and let $M:=\bT\-\Term(\tup{x}.\phi)/{\approx_\bT}$.
    \begin{enumerate}
        \item\label{lem:basic_property_repn_model-1}
        For any Horn formula-in-context $\tup{y}.\psi$, 
        \begin{equation*}
            \intpn{\tup{y}.\psi}{M} = \{[\tup{\tau}]_\bT\mid\phi\seq{\tup{x}}\psi(\tup{\tau}/\tup{y})\text{ is a PHL-theorem of }\bT\}.
        \end{equation*}
        \item
        $M$ is a $\bT$-model.
        \item
        $[\tup{x}]_\bT=([x_1]_\bT,\dots,[x_n]_\bT) \in \intpn{\tup{x}.\phi}{M}$ holds.
    \end{enumerate}
\end{lemma}
\begin{proof}
    The proof is straightforward.
\end{proof}

\begin{definition}\label{def:representing_model}
    Let $\bT$ be a partial Horn theory over an $S$-sorted signature $\Sigma$.
    For each Horn formula-in-context $\tup{x}.\phi$ over $\Sigma$, define
    \begin{equation*}
        \repn{\tup{x}.\phi}_\bT:=\bT\-\Term(\tup{x}.\phi)/{\approx_{\bT}}
        \quad\in\bT\-\PMod.
    \end{equation*}
    This $\repn{\tup{x}.\phi}_\bT$ is called the \emph{representing $\bT$-model for $\tup{x}.\phi$}.
\end{definition}

\begin{proposition}\label{prop:repn_obj_represents_intpn}
    Let $\bT$ be a partial Horn theory over an $S$-sorted signature $\Sigma$.
    Let $\tup{x}.\phi$ be a Horn formula-in-context over $\Sigma$.
    Then the representing model $\repn{\tup{x}.\phi}_\bT$ represents the interpretation functor $\intpn{\tup{x}.\phi}{\bullet}:\bT\-\PMod\ni M\mapsto\intpn{\tup{x}.\phi}{M}\in\Set$,
    i.e., for every $\bT$-model $M$, we have the following natural isomorphism:
    \begin{equation*}
        \intpn{\tup{x}.\phi}{M} \cong \bT\-\PMod(\repn{\tup{x}.\phi}_\bT,M).
    \end{equation*}
\end{proposition}
\begin{proof}
    Take a $\bT$-model $M$ and a $\Sigma$-homomorphism $h:\repn{\tup{x}.\phi}_\bT\to M$.
    By induction on the structure of $\bT$-term $[\tau]_\bT$, $h([\tau]_\bT)$ can only be $\intpn{\tup{x}.\tau}{M}(h([\tup{x}]_\bT))$.
    Conversely, given $\tup{m}\in\intpn{\tup{x}.\phi}{M}$, 
    $h([\tau]_\bT):=\intpn{\tup{x}.\tau}{M}(\tup{m})$ yields a well-defined $\Sigma$-homomorphism $h:\repn{\tup{x}.\phi}_\bT\to M$.
    Thus, a $\Sigma$-homomorphism $h:\repn{\tup{x}.\phi}_\bT\to M$ corresponds bijectively to $\tup{m}\in\intpn{\tup{x}.\phi}{M}$.
\end{proof}

\begin{corollary}\label{cor:morphism_between_fp}
    Let $\bT$ be a partial Horn theory over an $S$-sorted signature $\Sigma$ and let $\tup{x}.\phi$ and $\tup{y}.\psi$ be Horn formulas over $\Sigma$ with $\tup{x}=(x_1,\dots,x_n)$ and $\tup{y}=(y_1,\dots,y_m)$.
    Then the following data correspond bijectively to each other:
    \begin{enumerate}
        \item
        A $\Sigma$-homomorphism $h:\repn{\tup{x}.\phi}_\bT\to\repn{\tup{y}.\psi}_\bT$,
        \item
        Equivalence classes $[\tau_1]_\bT,\dots,[\tau_n]_\bT$ of $\bT$-terms generated by $\tup{y}.\psi$ such that $\psi\seq{\tup{y}}\phi(\tup{\tau}/\tup{x})$ is a PHL-theorem of $\bT$.
    \end{enumerate}
\end{corollary}
\begin{proof}
    By \cref{prop:repn_obj_represents_intpn} and \cref{lem:basic_property_repn_model}\ref{lem:basic_property_repn_model-1}, we have:
    \begin{align*}
        \bT\-\PMod(\repn{\tup{x}.\phi}_\bT , \repn{\tup{y}.\psi}_\bT)
        &\cong \intpn{\tup{x}.\phi}{\repn{\tup{y}.\psi}_\bT}
        \\
        &= \{[\tup{\tau}]_\bT \mid \psi\seq{\tup{y}}\phi(\tup{\tau}/\tup{x})\text{ is a PHL-theorem of }\bT\}.
    \end{align*}
    This completes the proof.
\end{proof}

\begin{notation}
    Denote by
    \[
    \repn{\tau_1,\dots,\tau_n}_\bT:\repn{\tup{x}.\phi}_\bT\to\repn{\tup{y}.\psi}_\bT
    \]
    the morphism corresponding to $\bT$-terms $\tau_1,\dots,\tau_n$ by \cref{cor:morphism_between_fp}.
\end{notation}

\begin{remark}\label{rem:filtered_colim_partial_model}
    Given a filtered diagram $M_\bullet :\bI\to\bT\-\PMod$, let us construct a colimit $N=\Colim{I\in\bI}M_I$ in $\bT\-\PMod$.
    For each sort $s\in S$ define $N_s:=\Colim{I\in\bI}(M_I)_s$ as a colimit in $\Set$, i.e., 
    $N_s$ is the quotient set of the disjoint union $\coprod_{I\in\bI}(M_I)_s=\{ (I;x) \mid I\in\bI, x\in (M_I)_s \}$ by an equivalence relation $\sim_s$.
    Here $(I;x)\sim_s (J;y)$ holds if and only if there exists a span
    $
    \begin{tikzcd}
        I\arrow[r,"p"] & K & J\arrow[l,"q"']
    \end{tikzcd}
    $
    in $\bI$ satisfying $M_p(x)=M_q(y)$.
    Let $[I;x]$ denote the equivalence class with respect to $\sim_s$ containing $(I;x)$.
    
    Now the $S$-sorted set $N=(N_s)_{s\in S}$ yields a partial $\Sigma$-structure as follows:
    \begin{itemize}
        \item
        For each function symbol $f:s_1\times\dots\times s_n\to s$ in $\Sigma$, $\intpn{f}{N}([I_1;x_1],\dots,[I_n;x_n])$ is defined if and only if there exists a cocone $(I_i\longarr{p_i}I)_{1\le i\le n}$ such that $(M_{p_1}(x_1),\dots,M_{p_n}(x_n))$ belongs to the domain of $\intpn{f}{M_I}$, and then define
        \[
        \intpn{f}{N}([I_1;x_1],\dots,[I_n;x_n]):=[I;\intpn{f}{M_I}(M_{p_1}(x_1),\dots,M_{p_n}(x_n))];
        \]
        \item
        For each relation symbol $R:s_1\times\dots\times s_n$ in $\Sigma$, $([I_1;x_1],\dots,[I_n;x_n])$ belongs to $\intpn{R}{N}$ if and only if there exists a cocone $(I_i\longarr{p_i}I)_{1\le i\le n}$ such that $(M_{p_1}(x_1),\dots,M_{p_n}(x_n))$ belongs to $\intpn{R}{M_I}$.
    \end{itemize}
    Since each $M_I$ is $\bT$-model, we see that $N$ is a $\bT$-model and a colimit of the diagram $M_\bullet$.
\end{remark}

\begin{theorem}\label{thm:repn_enumerates_fpobjects}
    Let $\bT$ be a partial Horn theory over an $S$-sorted signature $\Sigma$.
    Then, for any $\bT$-model $M$, the following are equivalent:
    \begin{enumerate}
        \item
        $M$ is a finitely presentable object in $\bT\-\PMod$.
        \item\label{thm:repn_enumerates_fpobjects-2}
        There exists a Horn formula-in-context $\tup{x}.\phi$ such that $M\cong\repn{\tup{x}.\phi}_\bT$ in $\bT\-\PMod$.
    \end{enumerate}
\end{theorem}
\begin{proof}
Let $\fp{\bT\-\PMod}$ denote the class of all finitely presentable objects in $\bT\-\PMod$, and let $\mathcal{C}$ denote the class of all $\bT$-models satisfying \ref{thm:repn_enumerates_fpobjects-2}.
We first observe $\mathcal{C}\subseteq\fp{\bT\-\PMod}$.
By \cref{prop:repn_obj_represents_intpn}, this assertion is equivalent to saying that for every Horn formula-in-context $\tup{x}.\phi$, the interpretation functor $\intpn{\tup{x}.\phi}{\bullet}:\bT\-\PMod\to\Set$ preserves filtered colimits, which follows from the construction of filtered colimits in $\bT\-\PMod$.

We next consider the following set $G$ of atomic formulas-in-context:
\begin{multline*}
    G:=
    \{ (x{:}s).\top \mid s\in S \}
    \cup
    \{ \tup{x}.f(\tup{x})\defined \mid f\text{ is a function symbol in }\Sigma. \}
    \\
    \cup
    \{ \tup{x}.R(\tup{x}) \mid R\text{ is a relation symbol in }\Sigma.\}.
\end{multline*}
Then $\mathcal{G}:=\{\repn{\tup{x}.\phi}_\bT \mid \tup{x}.\phi\in G\}$ forms a (small) strong generator for $\bT\-\PMod$, that is, the functor $\bT\-\PMod\in M\mapsto \{\bT\-\PMod(\repn{\tup{x}.\phi}_\bT,M)\}_{\tup{x}.\phi\in G}\cong
\{\intpn{\tup{x}.\phi}{M}\}_{\tup{x}.\phi\in G}\in\Set^G$ is conservative and faithful.
From $\mathcal{G}\subseteq\mathcal{C}\subseteq\fp{\bT\-\PMod}$ it follows that $\fp{\bT\-\PMod}$ is the finite colimit closure of $\mathcal{G}$ by \cref{thm:basic_properties_of_lfp}\ref{thm:basic_properties_of_lfp-1} since $\bT\-\PMod$ is locally finitely presentable by \cref{thm:PHT_characterize_lfp}.
So it suffices to prove that $\mathcal{C}$ is closed under finite colimits in $\bT\-\PMod$.

We first prove the case of finite coproducts.
Let $M_k\,(1\le k\le n)$ be $\bT$-models belonging to $\mathcal{C}$.
For each $k$, take a Horn formula-in-context $\tup{x}_k.\phi_k$ satisfying $M_k\cong\repn{\tup{x}_k.\phi_k}_\bT$.
There is no loss of generality in assuming that $\tup{x}_k$ and $\tup{x}_{k'}$ contain no common variable when $k\neq k'$.
Then for every $\bT$-model $A$, we have the following by \cref{prop:repn_obj_represents_intpn}:
\begin{align*}
    \bT\-\PMod\left(\repn{(\tup{x}_k)_{1\le k\le n}.\bigwedge_{1\le k\le n}\phi_k}_\bT,A\right)
    &\cong\intpn{(\tup{x}_k)_{1\le k\le n}.\bigwedge_{1\le k\le n}\phi_k}{A}
    \\
    &\cong\prod_{1\le k\le n}\intpn{\tup{x}_k.\phi_k}{A}
    \\
    &\cong\prod_{1\le k\le n}\bT\-\PMod(\repn{\tup{x}_k.\phi_k}_\bT,A).
\end{align*}
This implies that $\repn{(\tup{x}_k)_{1\le k\le n}.\bigwedge_{1\le k\le n}\phi_k}_\bT$ is a coproduct of $M_k\,(1\le k\le n)$ and belongs to $\mathcal{C}$.

We next prove the case of coequalizers.
Consider parallel morphisms
\begin{tikzcd}
    M\arrow[r,"h",shift left]\arrow[r,"h'"',shift right] & N,
\end{tikzcd}
where $M,N\in\mathcal{C}$.
Take Horn formulas $(x_1,\dots,x_n).\phi$ and $(y_1,\dots,y_m).\psi$ satisfying $M\cong\repn{\tup{x}.\phi}_\bT$ and $N\cong\repn{\tup{y}.\psi}_\bT$ respectively.
For simplicity, we assume $M=\repn{\tup{x}.\phi}_\bT$ and $N=\repn{\tup{y}.\psi}_\bT$.
By \cref{cor:morphism_between_fp}, $h$ and $h'$ are written as $h=\repn{\tup{\tau}}_\bT$ and $h'=\repn{\tup{\tau}'}_\bT$.
Considering a Horn formula $\chi:=\left( \psi\wedge\bigwedge_{1\le i\le n}\tau_i=\tau'_i \right)$, we have
\begin{align*}
    \bT\-\PMod( \repn{\tup{y}.\chi}_\bT , A )
    &\cong \intpn{\tup{y}.\chi}{A}
    \\
    &=  \{ \tup{a}\in\intpn{\tup{y}.\psi}{A} \mid \forall i,\,\intpn{\tup{y}.\tau_i}{A}(\tup{a})=\intpn{\tup{y}.\tau'_i}{A}(\tup{a}) \}
    \\
    &\cong \{ g\in\bT\-\PMod(\repn{\tup{y}.\psi}_\bT,A) \mid gh=gh' \}
\end{align*}
for every $\bT$-model $A$ by \cref{prop:repn_obj_represents_intpn}.
This implies that $\repn{\tup{y}.\chi}_\bT$ is a coequalizer of $h$ and $h'$, and belongs to $\mathcal{C}$.
\end{proof}

\begin{proposition}\label{prop:validity_for_PHL}
    Let $\bT$ be a partial Horn theory over an $S$-sorted signature $\Sigma$.
    Let $\phi\seq{\tup{x}}\psi$ be a Horn sequent.
    Then the morphism
    \[
    \repn{\tup{x}.\phi}_\bT\longarr{\repn{\tup{x}}_\bT}\repn{\tup{x}.\phi\wedge\psi}_\bT \incat{\bT\-\PMod}
    \]
    is an epimorphism, and for any $\bT$-model $M$, the following are equivalent:
    \begin{enumerate}
        \item
        $\phi\seq{\tup{x}}\psi$ is valid in $M$,
        \item\label{prop:validity_for_PHL-2}
        $M$ is orthogonal to $\repn{\tup{x}}_\bT:\repn{\tup{x}.\phi}_\bT\to\repn{\tup{x}.\phi\wedge\psi}_\bT$, 
        i.e., for any $\Sigma$-homomorphism $h:\repn{\tup{x}.\phi}_\bT\to M$, there exists a (necessarily unique) $\Sigma$-homomorphism $\hat{h}:\repn{\tup{x}.\phi\wedge\psi}_\bT\to M$ satisfying $\hat{h}\circ\repn{\tup{x}}_\bT=h$.
    \end{enumerate}
    \begin{equation*}
        \begin{tikzcd}
            \repn{\tup{x}.\phi}_\bT\arrow[r,"h"]\arrow[d,"\repn{\tup{x}}_\bT"'] & M \\
            \repn{\tup{x}.\phi\wedge\psi}_\bT\arrow[ru,"\exists!\hat{h}"',dashed] &
        \end{tikzcd}
    \end{equation*}
\end{proposition}
\begin{proof}
    \cref{prop:repn_obj_represents_intpn} ensures that $\repn{\tup{x}}_\bT$ is an epimorphism, hence the uniqueness of $\hat{h}$ in \ref{prop:validity_for_PHL-2} always holds.
    The existence of $\hat{h}$ for every $h$ is clearly equivalent to $\intpn{\tup{x}.\phi}{M}\subseteq\intpn{\tup{x}.\phi\wedge\psi}{M}$.
\end{proof}

\section{Relative algebraic theories}\label{section3}
We fix an $S$-sorted signature $\Sigma$ and a partial Horn theory $\bS$ over $\Sigma$ throughout this section, and define \emph{$\bS$-relative algebraic theories}.
Note that we define $\bS$-relative algebraic theories, not ``$\A(\simeq\bS\-\PMod)$-relative algebraic theories''.
However, we may call an $\bS$-relative algebraic theory an ``$\A$-relative algebraic theory'' by \cref{cor:S-rat_is_independent_of_S}.
Later, we will show that $\bS$-relative algebraic theories correspond with finitary monads on $\bS\-\PMod$ (\cref{thm:equiv_between_monad_and_rat}).
We only show one direction of such correspondence here.

\subsection{Relative algebras}
In an ordinary multi-sorted algebraic theory, each operator has an arity as a finite tuple of sorts, which represents its domain.
We generalize the concept of arity from finite tuple of sorts to Horn formula-in-context.
Since Horn formulas-in-context correspond to finitely presentable objects by \cref{thm:repn_enumerates_fpobjects}, this generalization is closely related to \cite{kelly1993adjunctions} and \cite{rosicky2021metric}.

\begin{definition}\quad
\begin{enumerate}
    \item
    An \emph{$\bS$-relative signature} $\Omega$ is a set $\Omega$ such that for each element $\omega\in\Omega$, a Horn formula-in-context $\tup{x}.\phi$ over $\Sigma$ and a sort $s\in S$ are given.
    The Horn formula-in-context $\tup{x}.\phi$ is called an \emph{arity} of $\omega$ and written as $\ar(\omega)$.
    The sort $s$ is called a \emph{type} of $\omega$ and written as $\type(\omega)$.
    \item
    Given an $\bS$-relative signature $\Omega$, 
    each $\omega\in\Omega$ can be regarded as a function symbol $\omega:s_1\times\dots\times s_n\to s$ if $\type(\omega)=s$ and $\ar(\omega)$ is in the context $x_1{:}s_1,\dots,x_n{:}s_n$.
    Denote by $\Sigma+\Omega$ the $S$-sorted signature obtained by adding to $\Sigma$ all $\omega\in\Omega$ in this way.
    A Horn sequent $\phi\seq{\tup{x}}\psi$ over $\Sigma+\Omega$ is called an \emph{$\bS$-relative judgment} if $\phi$ is over $\Sigma$, i.e., if $\phi$ contains no function symbol derived from $\Omega$.
    \item
    An \emph{$\bS$-relative algebraic theory} is a pair $(\Omega,E)$ of an $\bS$-relative signature $\Omega$ and a set $E$ of $\bS$-relative judgments.
\end{enumerate}
\end{definition}

\begin{definition}
    Let $\Omega$ be an $\bS$-relative signature.
    An ($\bS$-relative) \emph{$\Omega$-algebra} $\bA$ consists of:
    \begin{itemize}
        \item a partial $\bS$-model $A$,
        \item for each $\omega\in\Omega$, a map $\intpn{\omega}{\bA}: \intpn{\ar(\omega)}{A}\to A_{\type(\omega)}$.
    \end{itemize}
\end{definition}

An $\Omega$-algebra $\bA$ can be regarded as a partial $(\Sigma+\Omega)$-structure by considering $\intpn{\omega}{\bA}$ as a partial map $A_{s_1}\times\dots\times A_{s_n}\pto A_{\type(\omega)}$, where $\ar(\omega)$ is in the context $x_1{:}s_1,\dots,x_n{:}s_n$.
Conversely, a partial $(\Sigma+\Omega)$-structure satisfying all sequents in $\bS$ and the bisequent $\omega(\tup{x})\defined\biseq{\tup{x}}\ar(\omega)$ for each $\omega\in\Omega$ can be regarded as an $\Omega$-algebra.

\begin{definition}
    Let $\Omega$ be an $\bS$-relative signature.
    We say an $\Omega$-algebra $\bA$ satisfies an $\bS$-relative judgment if $\bA$ satisfies it as a $(\Sigma+\Omega)$-structure.
\end{definition}

\begin{notation}\quad
\begin{enumerate}
    \item
    Given an $\bS$-relative signature $\Omega$, we will denote by $\Alg\Omega$ the full subcategory of $(\Sigma+\Omega)\-\PStr$ consisting of all $\Omega$-algebras.
    \item
    Given an $\bS$-relative algebraic theory $(\Omega,E)$, we will denote by $\Alg(\Omega,E)$ the full subcategory of $\Alg\Omega$ consisting of all algebras satisfying all $\bS$-relative judgments in $E$.
    An $\Omega$-algebra belonging to $\Alg(\Omega,E)$ is called an ($\bS$-relative) \emph{$(\Omega,E)$-algebra} or a \emph{model} of $(\Omega,E)$.
\end{enumerate}
\end{notation}

Let $(\Omega,E)$ be an $\bS$-relative algebraic theory.
Considering the following partial Horn theory over $\Sigma+\Omega$
\begin{equation}\label{eq:PHT_for_OmegaE}
    \pht{\Omega}{E}:=\bS\cup\{\omega(\tup{x})\defined\biseq{\tup{x}}\ar(\omega)\}_{\omega\in\Omega}\cup E,
\end{equation}
we have $\Alg(\Omega,E)\cong\pht{\Omega}{E}\-\PMod$.
Thus, the category $\Alg(\Omega,E)$ of relative algebras is the category of models of a partial Horn theory.
In particular, $\Alg(\Omega,E)$ is locally finitely presentable.

\begin{example}
    Let $\bS_\pos$ be the partial Horn theory for posets.
    We present an $\bS_\pos$-relative algebraic theory $(\Omega,E)$.
    \begin{gather*}
        \Omega:=\{-\},\quad \ar(-):=(x,y).y\le x, \\
        E:=\{\, x\le y \seq{x,y,z}(x-z)\le(y-z),\quad y\le z \seq{x,y,z}(x-z)\le(x-y) \,\}.
    \end{gather*}
    Then an $(\Omega,E)$-algebra is just a ``poset with subtraction.''
    For example, $\bN$ with usual subtraction is an $(\Omega,E)$-algebra.
\end{example}

\begin{example}
    The partial Horn theory $\bS_\quiv$ for quivers is given by:
    \begin{gather*}
        S_\quiv:=\{\mathrm{e},\mathrm{v}\},\quad
        \Sigma_\quiv:=\{ \mathrm{s},\mathrm{t}:\mathrm{e}\to\mathrm{v} \},\quad
        \bS_\quiv:=\{ \top\seq{f:\mathrm{e}}\mathrm{s}(f)\defined\wedge\mathrm{t}(f)\defined \}.
    \end{gather*}
    We define an $\bS_\quiv$-relative algebraic theory $(\Omega,E)$ as follows:
    \begin{gather*}
        \Omega:=\{ \circ,\id \};\\
        \ar(\circ):=(g,f{:}\mathrm{e}).\mathrm{s}(g)=\mathrm{t}(f),\quad \type(\circ):=\mathrm{e};\\
        \ar(\id):=(x{:}\mathrm{v}).\top,\quad \type(\id):=\mathrm{e};
    \end{gather*}
    \begin{equation*}
        E:=\left\{
        \begin{gathered}
            \top\seq{x:\mathrm{v}}\mathrm{s}(\id(x))=x\wedge\mathrm{t}(\id(x))=x,\\
            \mathrm{s}(g)=\mathrm{t}(f)\seq{g,f:\mathrm{e}}\mathrm{s}(g\circ f)=\mathrm{s}(f)\wedge\mathrm{t}(g\circ f)=\mathrm{t}(g),\\
            \top\seq{f:\mathrm{e}}f\circ\id(\mathrm{s}(f))=f\wedge \id(\mathrm{t}(f))\circ f=f,\\
            \mathrm{s}(h)=\mathrm{t}(g)\wedge\mathrm{s}(g)=\mathrm{t}(f)\seq{h,g,f:\mathrm{e}}(h\circ g)\circ f=h\circ (g\circ f)
        \end{gathered}
        \right\}
    \end{equation*}
    Then, we have $\Alg(\Omega,E)\simeq\Cat$.
\end{example}

\begin{example}
    The partial Horn theory $\bS_\rsrel$ for reflexive symmetric relations is given by:
    \begin{gather*}
        S_\rsrel:=\{ * \},\quad
        \Sigma_\rsrel:=\{ \odot: *\times * \},\quad
        \bS_\rsrel:=\{ \top\seq{x} x\odot x,\quad x\odot y\seq{x,y}y\odot x \}.
    \end{gather*}
    We define an $\bS_\rsrel$-relative algebraic theory $(\Omega,E)$ as follows:
    \begin{gather*}
        \Omega:=\{ 0,1,\neg,\vee,\wedge \};\\
        \ar(0)=\ar(1):=().\top,\quad \ar(\neg):=x.\top,\quad \ar(\vee)=\ar(\wedge):=(x,y).x\odot y;
    \end{gather*}
    \begin{equation*}
        E:=\left\{
        \begin{gathered}
            \top \seq{x} x\odot 0,~x\odot 1;\quad\quad
            x\odot y \seq{x,y} x\odot\neg y;\\
            x\odot y,~y\odot z,~z\odot x \seq{x,y,z} x\odot (y\vee z),~x\odot (y\wedge z);\\
            x\odot y,~y\odot z,~z\odot x \seq{x,y,z} (x\vee y)\vee z=x\vee (y\vee z),~(x\wedge y)\wedge z=x\wedge (y\wedge z);\\
            x\odot y \seq{x,y} x\vee y=y\vee x,~x\wedge y=y\wedge x;\\
            x\odot y \seq{x,y} (x\wedge y)\vee x=x,~x\wedge (y\vee x)=x;\\
            \top \seq{x} x\vee 0=x,~x\wedge 1=x,~x\vee\neg x=1,~x\wedge\neg x=0;\\
            x\odot y,~y\odot z,~z\odot x \seq{x,y,z} (x\wedge y)\vee z=(x\vee z)\wedge (x\vee z);\\
            x\odot y,~y\odot z,~z\odot x \seq{x,y,z} (x\vee y)\wedge z=(x\wedge z)\vee (y\wedge z)
        \end{gathered}
        \right\}
    \end{equation*}
    In the above, we use the symbol (,) instead of $\wedge$ to avoid confusion.
    An algebra of $(\Omega,E)$ is a Boolean algebra-like algebra whose conjunction and disjunction are partial, which is called a \emph{partial Boolean algebra} in \cite{berg2012noncomm}.
    There, the reflexive symmetric relation $\odot$ is called \emph{commeasurability}.
\end{example}

\subsection{Morphisms of relative algebraic theories}

In this subsection, we define the category of $\bS$-relative algebraic theories.
A morphism defined here is a special case of a looser version of a theory morphism between partial Horn theories in \cite[Definition 9.10]{palmgren2007partial}.

\begin{definition}
    Let $(\Omega,E)$ and $(\Omega',E')$ be $\bS$-relative algebraic theories.
    \begin{enumerate}
        \item
        A \emph{theory morphism} $\rho:(\Omega,E)\to (\Omega',E')$ is an assignment to each operator $\omega\in\Omega$, a $\pht{\Omega'}{E'}$-term $\omega^\rho$ of type $\type(\omega)$ generated by $\ar(\omega)$ such that for every $(\phi\seq{\tup{x}}\psi)\in E$, $\phi\seq{\tup{x}}\psi^\rho$ is a PHL-theorem of $\pht{\Omega'}{E'}$.
        Here, $\psi^\rho$ is the $\rho$-translation, which is constructed by replacing all symbols that $\psi$ includes by $\rho$.
        \item
        Let $\rho,\sigma:(\Omega,E)\to (\Omega',E')$ be theory morphisms.
        We say that \emph{$\rho$ and $\sigma$ are equivalent} and write $\rho\sim\sigma$ if $\phi\seq{\tup{x}}\omega^\rho=\omega^\sigma$ is a PHL-theorem of $\pht{\Omega'}{E'}$ for every $\omega\in\Omega$ with $\ar(\omega)=\tup{x}.\phi$.
        \item
        Let $\rho:(\Omega,E)\to (\Omega',E')$ be a theory morphism.
        Given an algebra $\bA\in\Alg(\Omega',E')$, an algebra $\bA^\rho\in\Alg(\Omega,E)$ is defined by $\intpn{\omega}{\bA^\rho}:=\intpn{\omega^\rho}{\bA}$ for each $\omega\in\Omega$.
        Then, there exists a unique functor $\Alg\rho:\Alg(\Omega',E')\to\Alg(\Omega,E)$ such that $\bA\mapsto\bA^\rho$ and the following commutes:
        \begin{equation*}
            \begin{tikzcd}
                \Alg(\Omega,E)\arrow[rd,"U"'] & & \Alg(\Omega',E')\arrow[ll,"\Alg\rho"']\arrow[ld,"U'"] \\
                & \bS\-\PMod &
            \end{tikzcd}
        \end{equation*}
        Here, $U$ and $U'$ are forgetful functors.
    \end{enumerate}
\end{definition}

\begin{remark}\label{rem:equivalence_of_theory_morphism}
    Let $\rho,\sigma:(\Omega,E)\to (\Omega',E')$ be theory morphisms.
    Then, $\Alg\rho=\Alg\sigma$ if and only if $\rho\sim\sigma$.
\end{remark}

\begin{definition}\quad
    \begin{enumerate}
        \item
        We now define the category $\Th^\bS$ of $\bS$-relative algebraic theories:
        $\Th^\bS$ is the category whose object is an $\bS$-relative algebraic theory $(\Omega,E)$ and whose morphism from $(\Omega,E)$ to $(\Omega',E')$ is an equivalence class $[\rho]$ of a theory morphism $\rho:(\Omega,E)\to (\Omega',E')$.
        Given two morphisms $[\rho]:(\Omega,E)\to (\Omega',E')$ and $[\sigma]:(\Omega',E')\to (\Omega'',E'')$, the composition is an equivalence class of the theory morphism $\sigma\circ\rho$ which assigns to $\omega\in\Omega$, the $\pht{\Omega''}{E''}$-term $(\omega^\rho)^\sigma$.
        The identity morphism is an equivalence class of the identity theory morphism $\rho$ such that $\omega^\rho:=\omega$.
        \item
        We can define the functor $\Alg:\Th^\bS\to (\CAT/\bS\-\PMod)^\op$ such that $(\Omega,E)\mapsto\Alg(\Omega,E)$ and $[\rho]\mapsto\Alg\rho$, where $\CAT$ is the category of (not necessarily small) categories and $\CAT/\bS\-\PMod$ is the slice category.
        This functor is well-defined by \cref{rem:equivalence_of_theory_morphism}.
    \end{enumerate}
\end{definition}

\begin{theorem}\label{thm:alg_fully_faithful}
    The functor $\Alg:\Th^\bS \to (\CAT/\bS\-\PMod)^\op$ is fully faithful.
\end{theorem}
\begin{proof}
    By \cref{rem:equivalence_of_theory_morphism}, the functor $\Alg$ is faithful.
    To prove fullness, take a functor $K:\Alg(\Omega',E')\to\Alg(\Omega,E)$ which commutes with forgetful functors.
    Let $\omega\in\Omega$ with $\ar(\omega)=(x_1{:}s_1,\dots,x_n{:}s_n).\phi$ and $\type(\omega)=s$ and let $\bA:=\repn{\tup{x}.\phi}_\pht{\Omega'}{E'}\in\Alg(\Omega',E')$ and $A:=U'\bA$.
    Considering the interpretation $\intpn{\omega}{K\bA}:\intpn{\tup{x}.\phi}{A}\to A_s$, we get an equivalence class $\intpn{\omega}{K\bA}([\tup{x}]_\pht{\Omega'}{E'})\in A_s$ of a $\pht{\Omega'}{E'}$-term of type $s$ generated by $\tup{x}.\phi$ from the construction of $\repn{\tup{x}.\phi}_\pht{\Omega'}{E'}$ in \cref{def:representing_model}.
    We now define $\omega^\rho$ as a representative of the class $\intpn{\omega}{K\bA}([\tup{x}]_\pht{\Omega'}{E'})\in A_s$.

    We now show that the interpretation maps $\intpn{\omega}{K\bB},\intpn{\omega^\rho}{\bB}:\intpn{\tup{x}.\phi}{B}\to B_s$ coincide for every $\bB\in\Alg(\Omega',E')$.
    Write $B:=U'\bB$.
    Then, for every morphism $f:\bA\to\bB$ in $\Alg(\Omega',E')$, the following diagram is pairwise commutative:
    \begin{equation}\label{eq:pairwise_commutative_AB}
        \begin{tikzcd}
            A_{s_1}\times\dots\times A_{s_n}\arrow[d,"(U'f)_{s_1}\times\dots\times(U'f)_{s_n}"']\arrow[r,phantom,"\supseteq"] &[-10pt] \intpn{\tup{x}.\phi}{A}\arrow[d,dashed,"\intpn{\tup{x}.\phi}{U'f}"']\arrow[r,shift left=1,"\intpn{\omega}{K\bA}"]\arrow[r,shift right=1,"\intpn{\omega^\rho}{\bA}"'] & A_s\arrow[d,"(U'f)_s"] \\
            B_{s_1}\times\dots\times B_{s_n}\arrow[r,phantom,"\supseteq"] &[-10pt] \intpn{\tup{x}.\phi}{B}\arrow[r,shift left=1,"\intpn{\omega}{K\bB}"]\arrow[r,shift right=1,"\intpn{\omega^\rho}{\bB}"'] & B_s
        \end{tikzcd}\incat{\Set}
    \end{equation}
    By the definition of $\omega^\rho$, we have $\intpn{\omega}{K\bA}=\intpn{\omega^\rho}{\bA}$.
    By $\Alg(\Omega',E')(\bA,\bB)\cong\intpn{\tup{x}.\phi}{B}$, we see that every element of $\intpn{\tup{x}.\phi}{B}$ lies in the image of $\intpn{\tup{x}.\phi}{U'f}$ for some $f$.
    Thus, \cref{eq:pairwise_commutative_AB} implies $\intpn{\omega}{K\bB}=\intpn{\omega^\rho}{\bB}$.

    Consequently, we have $\bB^\rho=K\bB\in\Alg(\Omega,E)$ for every $\bB\in\Alg(\Omega',E')$, hence for every $(\phi\seq{\tup{x}}\psi)\in E$, $\phi\seq{\tup{x}}\psi^\rho$ is a PHL-theorem of $\pht{\Omega'}{E'}$.
    This yields a theory morphism $\rho:(\Omega,E)\to (\Omega',E')$ such that $K=\Alg\rho$.
\end{proof}

\subsection{From relative algebraic theories to finitary monads}\label{subsection3.2}
\begin{definition}
    A functor (or monad) is \emph{finitary} if it preserves filtered colimits.
\end{definition}

Our goal in this subsection is to prove that the category of models of an $\bS$-relative algebraic theory is (strictly) finitary monadic over $\bS\-\PMod$.
This is one direction of our main theorem (\cref{thm:equiv_between_monad_and_rat}).

\begin{notation}
    Given an endofunctor $H:\C\to\C$, we will denote by $\Alg H$ the inserter from $H$ to $\Id_\C$, i.e., 
    $\Alg H$ is the category whose object is a pair $(X,x)$ of:
    \begin{itemize}
        \item
        an object $X\in\C$ and
        \item
        a morphism $x:H(X)\to X$ in $\C$,
    \end{itemize}
    and whose morphism $f:(X,x)\to (Y,y)$ is a morphism $f:X\to Y$ in $\C$ such that the following commutes:
    \begin{equation*}
        \begin{tikzcd}
            H(X)\arrow[d,"x"']\arrow[r,"H(f)"] & H(Y)\arrow[d,"y"] \\
            X\arrow[r,"f"'] & Y
        \end{tikzcd}\incat{\C}.
    \end{equation*}   
\end{notation}

\begin{definition}\label{def:endofunc_for_signature}
    Given an $\bS$-relative signature $\Omega$, define a finitary endofunctor $H_\Omega:\bS\-\PMod\to\bS\-\PMod$ by the following:
    \begin{equation*}
        H_\Omega(A):=\coprod_{\omega\in\Omega}\intpn{\ar(\omega)}{A}\bullet\repn{x{:}\type(\omega).\top}_\bS.
    \end{equation*}
    Here, $\intpn{\ar(\omega)}{A}\bullet\repn{x{:}\type(\omega).\top}_\bS$ is the copower of $\repn{x{:}\type(\omega).\top}_\bS$ by the set $\intpn{\ar(\omega)}{A}$.
\end{definition}

\begin{lemma}\label{lem:iso_alg_alg}
    For every $\bS$-relative signature $\Omega$, there exists an isomorphism $\Alg H_\Omega\cong\Alg\Omega$ of categories that commutes with forgetful functors.
\end{lemma}
\begin{proof}
    For each $A\in\bS\-\PMod$, the following data correspond bijectively:
    \begin{center}
    \renewcommand{\arraystretch}{1.3}
    \begin{tabular}{c}
        $H_\Omega(A)\longarr{}A\incat{\bS\-\PMod}$
        \\
        \hline\hline
        $\intpn{\ar(\omega)}{A}\bullet\repn{x{:}\type(\omega).\top}_\bS \longarr{} A\incat{\bS\-\PMod}\quad(\omega\in\Omega)$
        \\
        \hline\hline
        $\intpn{\ar(\omega)}{A}\longarr{}\bS\-\PMod( \repn{x{:}\type(\omega).\top}_\bS , A )\incat{\Set}\quad(\omega\in\Omega)$
        \\
        \hline\hline
        $\intpn{\ar(\omega)}{A}\longarr{}A_{\type(\omega)}\incat{\Set}\quad(\omega\in\Omega)$
    \end{tabular}
    \renewcommand{\arraystretch}{1}
    \end{center}
    This gives a desired isomorphism of categories.
\end{proof}

\begin{lemma}\label{lem:finitary_endofunctor_monadic_adjoint}
    Let $H:\C\to\C$ be a finitary endofunctor on a cocomplete locally small category $\C$.
    Then the forgetful functor $U:\Alg H\to\C$ has a left adjoint, and this adjunction is strictly finitary monadic, i.e., the induced monad is finitary and the Eilenberg-Moore comparison functor is an isomorphism.
\end{lemma}
\begin{proof}
    We first construct a left adjoint of $U$.
    Fix $X\in\C$.
    Define objects $K_n\in\C$ and morphisms $k_n:K_n\to K_{n+1}$ in $\C$ inductively:
    \begin{itemize}
        \item $K_0:=X$,
        \item $K_{n+1}:=HK_n+X$,
        \item define $k_0:K_0=X\to HX+X=K_1$ to be the coprojection, and
        \item define $k_{n+1}:K_{n+1}=HK_n+X\to HK_{n+1}+X=K_{n+2}$ to be $Hk_n+\id$.
    \end{itemize}
    The above construction yields an $\bbomega$-indexed diagram $(K_n)_{n\in\bbomega}$ in $\C$.
    Since the diagram
    \begin{equation*}
        \begin{tikzcd}
            &[-20pt] HK_n\arrow[d,"\gamma_n"']\arrow[r,"Hk_n"] & HK_{n+1}\arrow[d,"\gamma_{n+1}"] &[-20pt] \\
            K_{n+1}\arrow[r,equal] & HK_n+X\arrow[r,"Hk_n+\id"'] & HK_{n+1}+X\arrow[r,equal] & K_{n+2}
        \end{tikzcd}
    \end{equation*}
    commutes, where $\gamma_n$ and $\gamma_{n+1}$ are coprojections, we have a canonical morphism 
    \begin{equation*}
        \gamma:H(\Colim{n\in\bbomega}K_n)\cong\Colim{n\in\bbomega}HK_n\longarr{\mathrm{Colim}\gamma_n}\Colim{n\in\bbomega}K_{n+1}.
    \end{equation*}
    Now we define $FX:=(K_\omega,\gamma)\in\Alg H$, where $K_\omega:=\Colim{n\in\bbomega}K_n$, and define $\eta_X:X\to K_\omega$ to be the zeroth coprojection.
    It is easy to check that $FX$ and $\eta_X$ yield a left adjoint $F$ of $U$ and its unit $\eta$.
    
    It remains to prove that $F\dashv U$ is finitary monadic.
    Since $H$ is finitary, $U$ creates filtered colimits.
    It is clear that $U$ strictly creates absolute coequalizers.
    Therefore, by Beck's strict monadicity theorem (\cite{maclane1998working} VI.7 Theorem 1), we have proved that $F\dashv U$ is strictly monadic.
\end{proof}

Combining \cref{lem:iso_alg_alg,lem:finitary_endofunctor_monadic_adjoint}, we obtain the following corollary.

\begin{corollary}\label{cor:alg_omega_monadic}
    Let $\Omega$ be an $\bS$-relative signature.
    Then $\Alg\Omega$ is a strictly finitary monadic category over $\bS\-\PMod$.
\end{corollary}

\vspace{10em}
\begin{definition}
    Let $\C$ be a category.
    \begin{enumerate}
        \item
        An object $C\in\C$ is \emph{orthogonal} to a morphism $f:X\to Y$ in $\C$ if for every morphism $g:X\to C$, there exists a unique $\hat{g}:Y\to C$ satisfying $\hat{g}f=g$.
        \begin{equation*}
            \begin{tikzcd}[column sep=large, row sep=large]
                X\arrow[d,"f"']\arrow[r,"g"] & C \\
                Y\arrow[ru,"\exists !\hat{g}"',dashed] &
            \end{tikzcd}
        \end{equation*}
        \item
        Given a class $\Lambda\subseteq\mor\C$ of morphisms, denote by $\orth{\Lambda}\subseteq\C$ the full subcategory consisting of all objects orthogonal to all morphisms in $\Lambda$.
        The full subcategory $\orth{\Lambda}$ is called an \emph{orthogonality class} of $\C$, and called a \emph{small-orthogonality class} if $\Lambda$ is (essentially) small.
    \end{enumerate}
\end{definition}

\begin{lemma}\label{lem:orth_closed_under_filcolim}
    Let $\C$ be a category with filtered colimits and let $\Lambda\subseteq\mor\C$ be a class of epimorphisms with finitely presentable domain.
    Then the orthogonality class $\orth{\Lambda}\subseteq\C$ is closed under filtered colimits.
\end{lemma}
\begin{proof}
    Let $C=\Colim{I\in\bI}C_I$ be a filtered colimit in $\C$ and assume $C_I\in\orth{\Lambda}$ for every $I\in\bI$.
    To prove $C\in\orth{\Lambda}$, take a morphism $f:X\to Y$ belonging to $\Lambda$ and a morphism $g:X\to C$.
    Since $X$ is finitely presentable, we have the expression $g:X\arr{g'}C_I\to C$ for some $I\in\bI$ and $g'$.
    By $C_I\in\orth{\Lambda}$, there exists a unique $h:Y\to C_I$ satisfying $hf=g'$, which implies that $C$ is orthogonal to $f$.
    \begin{equation*}
        \begin{tikzcd}[column sep=huge, row sep=huge]
            X\arrow[d,"f"']\arrow[r,"g"]\arrow[rd,"g'",dashed]  &  C  \\
            Y\arrow[r,"\exists !h",dashed]  &  C_I\arrow[u]
        \end{tikzcd}\incat{\C}
    \end{equation*}
\end{proof}

\begin{theorem}[{\cite[Theorem 5.4.7]{borceux1994handbook1}}]\label{thm:orth_sub_problem}
    Every small-orthogonality class of a locally presentable category is reflective, i.e., 
    the inclusion $\orth{\Lambda}\hookrightarrow\A$ has a left adjoint if $\A$ is locally presentable and $\Lambda\subseteq\mor\A$ is (essentially) small.
\end{theorem}

\begin{remark}\label{rem:alg_omegaE_orthogonal}
    Let $(\Omega,E)$ be an $\bS$-relative algebraic theory and let $\A:=\bS\-\PMod$.
    By \cref{cor:alg_omega_monadic}, we have the adjunction
    \begin{equation*}
    \begin{tikzcd}[column sep=large, row sep=large]
        \A\arrow[r,"\bF",shift left=7pt]\arrow[r,"\perp"pos=0.5,phantom] &[10pt]\Alg\Omega\arrow[l,"U",shift left=7pt]
    \end{tikzcd}
    \end{equation*}
    and $U$ preserves filtered colimits.
    Let $E=\{\phi_i\seq{\tup{x}_i}\psi_i\}_{i\in I}$.
    Since
    \begin{equation*}
        \Alg\Omega(\bF\repn{\tup{x}_i.\phi_i}_\bS , \bA)
        \cong \A(\repn{\tup{x}_i.\phi_i}_\bS , U\bA)
        \cong \intpn{\tup{x}_i.\phi_i}{U\bA}
        = \intpn{\tup{x}_i.\phi_i}{\bA}
    \end{equation*}
    holds naturally, we can assume $\bF\repn{\tup{x}_i.\phi_i}_\bS = \repn{\tup{x}_i.\phi_i}$ in $\Alg\Omega$.
    Here, $\repn{\tup{x}_i.\phi_i}$ is the abbreviation for $\repn{\tup{x}_i.\phi_i}_{\pht{\Omega}{\varnothing}}$, where $\pht{\Omega}{\varnothing}$ is the partial Horn theory for $\Alg\Omega$ as in \cref{eq:PHT_for_OmegaE}.
    By \cref{prop:validity_for_PHL}, $\Alg(\Omega,E)\subseteq\Alg\Omega$ is the full subcategory of objects orthogonal to the following:
    \begin{equation*}
        \Lambda:=\{~\bF\repn{\tup{x}_i.\phi_i}_\bS \arr{\repn{\tup{x}_i}} \repn{\tup{x}_i.\psi_i}~\}_{i\in I}.
    \end{equation*}
    Thus, \cref{thm:orth_sub_problem} gives the following adjunction:
    \begin{equation*}
    \begin{tikzcd}[column sep=large, row sep=large]
        \Alg\Omega\arrow[r,"r",shift left=7pt]\arrow[r,"\perp"pos=0.5,phantom] &[10pt]\Alg(\Omega,E)~(=\orth{\Lambda}).\arrow[l,shift left=7pt,hook']
    \end{tikzcd}
    \end{equation*}
\end{remark}

\begin{definition}
    Let $U:\C\to\A$ be a functor.
    A morphism $f$ in $\C$ is a \emph{$U$-retraction} if $Uf$ is a retraction, i.e., there exists a morphism $s$ in $\A$ such that $(Uf)s=\id$.
    Given a $U$-retraction $f:X\to Y$, $Y$ is called a \emph{$U$-retract} of $X$.
\end{definition}

\begin{lemma}\label{lem:closed_under_localret_filteredcolim}
    Let $(\Omega,E)$ be an $\bS$-relative algebraic theory and let $\A:=\bS\-\PMod$.
    Consider the forgetful functor $U:\Alg\Omega\to\A$.
    Then $\Alg(\Omega,E)\subseteq\Alg\Omega$ is closed under:
    \begin{enumerate}
        \item $U$-retracts and
        \item filtered colimits.
    \end{enumerate}
\end{lemma}
\begin{proof}
    We follow the notation used in \cref{rem:alg_omegaE_orthogonal}.
    \begin{enumerate}
        \item
        Let $p:\bA\to \bB$ in $\Alg\Omega$ be a $U$-retraction and assume $\bA\in\Alg(\Omega,E)$.
        To prove $\bB\in\Alg(\Omega,E)$, take a morphism $f:\bF\repn{\tup{x}_i.\phi_i}_\bS\to\bB$.
        Considering the morphism $f^\flat:\repn{\tup{x}_i.\phi_i}_\bS\to U\bB$ corresponding to $f$ by the adjunction $\bF\dashv U$, 
        since $Up$ is a retraction, there exists a morphism $g:\repn{\tup{x}_i.\phi_i}_\bS\to U\bA$ satisfying $(Up)g=f^\flat$.
        Considering the morphism $g^\sharp:\bF\repn{\tup{x}_i.\phi_i}_\bS\to \bA$ corresponding to $g$ by $\bF\dashv U$, 
        since $\phi_i\seq{\tup{x}_i}\psi_i$ is valid in $\bA$, there exists a unique $h:\repn{\tup{x}_i.\psi_i}\to \bA$ such that $h\repn{\tup{x}_i}=g^\sharp$.
        Then $ph\repn{\tup{x}_i}=pg^\sharp=f$ holds.
        Moreover, such $ph$ is unique since $\repn{\tup{x}_i}$ is an epimorphism.
        Therefore, $\bB$ is orthogonal to $\repn{\tup{x}_i}\,(\forall i\in I)$ and it follows from \cref{prop:validity_for_PHL} that $\bB$ belongs to $\Alg(\Omega,E)$.
        \begin{equation*}
        \begin{tikzcd}[column sep=huge, row sep=huge]
            \repn{\tup{x}_i.\phi_i}_\bS\arrow[r,"\exists g",dashed]\arrow[rd,"f^\flat"',pos=0.7]  &  U\bA\arrow[d,"Up"]  \\
            &  U\bB
        \end{tikzcd}\incat{\A}
        \quad\quad
        \begin{tikzcd}[column sep=huge, row sep=huge]
            \bF\repn{\tup{x}_i.\phi_i}_\bS\arrow[d,"\repn{\tup{x}_i}"']\arrow[r,"g^\sharp"]\arrow[rd,"f"',pos=0.7]  &  \bA\arrow[d,"p"]  \\
            \repn{\tup{x}_i.\psi_i}\arrow[ru,"\exists !h",pos=0.7,crossing over,dashed]  &  \bB
        \end{tikzcd}\incat{\Alg\Omega}
        \end{equation*}
        \item
        This follows from \cref{lem:orth_closed_under_filcolim}.
        \qedhere
    \end{enumerate}
\end{proof}

\begin{theorem}\label{thm:from_alg_to_monad}
    Let $(\Omega,E)$ be an $\bS$-relative algebraic theory.
    Then $\Alg(\Omega,E)$ is a strictly finitary monadic category over $\bS\-\PMod$.
\end{theorem}
\begin{proof}
    We follow the notation used in \cref{rem:alg_omegaE_orthogonal}.
    We have the following adjunctions:
    \begin{equation*}
    \begin{tikzcd}
        \A\arrow[r,"\bF",shift left=7pt]\arrow[r,"\perp"pos=0.5,phantom]
        &[10pt]\Alg\Omega\arrow[l,"U",shift left=7pt]\arrow[r,"r",shift left=7pt]\arrow[r,"\perp"pos=0.5,phantom]
        &[10pt]\Alg(\Omega,E)\arrow[l,shift left=7pt,hook',"\iota"]
    \end{tikzcd}
    \end{equation*}
    and $U\iota$ preserves filtered colimits.
    To prove that $U\iota$ is monadic, we use Beck's strict monadicity theorem (\cite{maclane1998working} VI.7 Theorem 1).
    We claim that $\Alg(\Omega,E)\subseteq\Alg\Omega$ is closed under $U$-split coequalizers, hence $U\iota$ strictly creates $U$-split coequalizers.
    Indeed, given a $U$-split coequalizer:
    \begin{equation*}
    \begin{tikzcd}[column sep=large, row sep=large]
        \bA\arrow[r,shift left=4pt,"f"]\arrow[r,shift right=4pt,"g"']  &  \bB\arrow[r,"q"]  &  \bC
    \end{tikzcd}\incat{\Alg\Omega}
    \end{equation*}
    with $\bA,\bB\in\Alg(\Omega,E)$, $Uq$ is a retraction and thus $q$ is a $U$-retraction.
    By \cref{lem:closed_under_localret_filteredcolim}, $\Alg(\Omega,E)\subseteq\Alg\Omega$ is closed under $U$-retracts, which shows that $\bC\in\Alg(\Omega,E)$.
\end{proof}

\section{A generalization of Birkhoff's variety theorem}\label{section4}
Our goal is to prove \cref{thm:birkhoff_for_rat}, which is a generalization of the classical Birkhoff's theorem (\cref{thm:birkhoff_classical}) to our relative algebraic theories.

\subsection{Presentable proper factorization systems}
We recall the definition of an orthogonal factorization system, which plays an important role in subsequent subsections.

\begin{definition} Let $\C$ be a category.
    \begin{enumerate}
        \item
        Given morphisms $e$ and $m$ in $\C$, we write $e\perp m$ if for any commutative square $ve=mu$ there exists a unique diagonal filler making both triangles commute:
        \begin{equation*}
            \begin{tikzcd}
                \cdot\arrow[r,"u"]\arrow[d,"e"'] & \cdot\arrow[d,"m"] \\
                \cdot\arrow[r,"v"']\arrow[ur,dashed,"\exists !"] & \cdot
            \end{tikzcd}
        \end{equation*}
        \item
        Given a class of morphisms $\Lambda$, denote by $\lorth(\Lambda)$ and $\rorth(\Lambda)$ the classes
        \begin{align*}
            \lorth(\Lambda)&:=\{ e \mid e\perp m\text{ for all }m\in\Lambda\} \\
            \rorth(\Lambda)&:=\{ m \mid e\perp m\text{ for all }e\in\Lambda\}.
        \end{align*}
        \item
        An \emph{orthogonal factorization system} in $\C$ is a pair $(\bE,\bM)$ of classes of morphisms in $\C$ that satisfies the following conditions:
        \begin{itemize}
            \item $\bE$ and $\bM$ are closed under composition and contain all isomorphisms in $\C$,
            \item Every morphism $f$ in $\C$ has a factorization $f=me$ with $e\in\bE$ and $m\in\bM$,
            \item $\bE\perp\bM$ holds, i.e., for any $e\in\bE$ and $m\in\bM$, $e\perp m$ holds.
        \end{itemize}
        Given an orthogonal factorization system $(\bE,\bM)$, we have $\bE=\lorth(\bM)$ and $\bM=\rorth(\bE)$.
        These can be verified straightforwardly.
        \item
        A \emph{proper factorization system} is an orthogonal factorization system $(\bE,\bM)$ such that every morphism in $\bE$ is an epimorphism and every morphism in $\bM$ is a monomorphism.
    \end{enumerate}
\end{definition}

The following is an orthogonal version of the fact known as the small object argument.

\begin{theorem}\label{thm:OFS_in_LPcategory}
    Let $\Lambda$ be a set of morphisms in a locally presentable category $\A$.
    Then, 
    \[
    (\,\lorth\rorth(\Lambda),\rorth(\Lambda)\,)
    \]
    is an orthogonal factorization system in $\A$.
\end{theorem}
\begin{proof}
    It suffices to show that every morphism in $\A$ has a $(\,\lorth\rorth(\Lambda),\rorth(\Lambda)\,)$-factorization.
    Fix $X\in\A$, and consider the following full subcategory of the slice category $\A/X$:
    \begin{equation*}
        \rorth(\Lambda)/X:=\{ m\in\A/X \mid m\in\rorth(\Lambda) \}.
    \end{equation*}
    $\rorth(\Lambda)/X$ is a small orthogonality class of $\A/X$.
    Indeed, $\orth{\Lambda_X}=\rorth(\Lambda)/X$ holds, where $\Lambda_X$ is the class of all morphisms in $\A/X$ belonging to $\Lambda$.
    Since $\A/X$ is locally presentable, $\rorth(\Lambda)/X$ is a reflective full subcategory of $\A/X$.
    Thus, for each $f:Y\to X$ in $\A$, we can take the following reflection into $\rorth(\Lambda)/X$:
    \begin{equation}\label{eq:reflection_to_M}
        \begin{tikzcd}
            Y\arrow[rr,"\eta_f"]\arrow[rd,"f"'] & & A_f\arrow[ld,"m_f\,\in\rorth(\Lambda)"] \\
            & X &
        \end{tikzcd}\incat{\A}
    \end{equation}

    To prove that $\eta_f$ belongs to $\lorth\rorth(\Lambda)$, 
    take the following commutative diagram arbitrarily:
    \begin{equation}\label{eq:comm_square_etaf}
        \begin{tikzcd}
            Y\arrow[d,"\eta_f"']\arrow[r,"u"] & B\arrow[d,"m\,\in\rorth(\Lambda)"] \\
            A_f\arrow[r,"v"'] & C
        \end{tikzcd}\incat{\A}.
    \end{equation}
    Let us construct a unique diagonal filler for the square \cref{eq:comm_square_etaf}.
    Take the pullback $P$ of $v$ and $m$, and consider a canonical morphism $g:Y\to P$ as follows:
    \begin{equation*}
        \begin{tikzcd}
            Y\arrow[rrd,bend left=15,"u"]\arrow[ddr,bend right=15,"\eta_f"']\arrow[rd,dashed,"g"',pos=0.8] &[-10pt] & \\[-10pt]
            & P\arrow[r,"\pi'"']\arrow[d,"\pi"]\arrow[rd,pos=0.1,phantom,"\lrcorner"] & B\arrow[d,"m"] \\
            & A_f\arrow[r,"v"'] & C
        \end{tikzcd}\incat{\A}
    \end{equation*}
    Since $g$ is a morphism $f\to m_f\pi$ in $\A/X$ and $\rorth(\Lambda)$ is closed under compositions and pullbacks, $m_f\pi$ belongs to $\rorth(\Lambda)$.
    Thus, by the universality of $\eta_f$, there exists a unique morphism $h:A_f\to P$ which makes the following commutes:
    \begin{equation}\label{eq:unique_h}
        \begin{tikzcd}
            Y\arrow[rrr,"g"]\arrow[rd,"\eta_f"description]\arrow[rddd,"f"'] & &[-20pt] &[-20pt] P\arrow[ddl,"\pi"] \\
            & A_f\arrow[rru,dashed,"h"]\arrow[dd,"m_f"description] & & \\[-20pt]
            & & A_f\arrow[ld,"m_f"] & \\
            & X & &
        \end{tikzcd}\incat{\A}.
    \end{equation}
    Then, in the slice category $\A/X$,
    \begin{equation*}
        \begin{tikzcd}
            f\arrow[rr,"\eta_f"]\arrow[dr,"\eta_f"'] & & m_f \\
            & m_f\arrow[ru,"\pi h"'] &
        \end{tikzcd}\incat{\A/X}
    \end{equation*}
    commutes, and thus $\pi h=\id$ holds by the universality of $\eta_f$.
    Now the following diagram commutes, and consequently, we get a diagonal filler $\pi' h$ for \cref{eq:comm_square_etaf}.
    \begin{equation*}
        \begin{tikzcd}
            Y\arrow[dd,"\eta_f"']\arrow[rr,"u"]\arrow[rd,"g"description] & & B\arrow[dd,"m"] \\
            & P\arrow[ru,"\pi'"description]\arrow[d,"\pi"] & \\
            A_f\arrow[ur,"h"description]\arrow[r,equal] & A_f\arrow[r,"v"'] & C
        \end{tikzcd}\incat{\A}
    \end{equation*}

    It remains to verify the uniqueness of $\pi' h$.
    Given another diagonal filler $k:A_f\to B$ for the square \cref{eq:comm_square_etaf}, 
    consider the following canonical morphism $h':A_f\to P$:
    \begin{equation*}
        \begin{tikzcd}
            A_f\arrow[rrd,bend left=15,"k"]\arrow[ddr,bend right=15,equal]\arrow[rd,dashed,"h'"',pos=0.8] &[-10pt] & \\[-10pt]
            & P\arrow[r,"\pi'"']\arrow[d,"\pi"]\arrow[rd,pos=0.1,phantom,"\lrcorner"] & B\arrow[d,"m"] \\
            & A_f\arrow[r,"v"'] & C
        \end{tikzcd}\incat{\A}.
    \end{equation*}
    By the universality of $P$, we have $h'\eta_f=g$, which implies the diagram replacing $h$ in \cref{eq:unique_h} with $h'$ commutes.
    Thus $h=h'$, and we have $k=\pi' h'=\pi' h$.

    The above argument proves the unique existence of a diagonal filler for \cref{eq:comm_square_etaf}, so $\eta_f$ belongs to $\lorth\rorth(\Lambda)$.
\end{proof}

\begin{notation}
    Given a class of morphisms $\Lambda$, denote by $\Lambda_\lambda$ the essentially small class
    \begin{equation*}
        \Lambda_\lambda:=\{ f\in\Lambda \mid \dom f, \cod f\text{ are }\lambda\text{-presentable} \}.
    \end{equation*}
\end{notation}

\begin{lemma}\label{lem:mono_rlp_wrt_retracts}
    Let $\A$ be a locally $\lambda$-presentable category.
    Denote by $\mathbf{Ret}$ the class of all retractions in $\A$ and by $\mathbf{Mono}$ the class of all monomorphisms in $\A$.
    Then, $\rorth(\mathbf{Ret}_\lambda)=\mathbf{Mono}$ holds.
\end{lemma}
\begin{proof}
    To show $\rorth(\mathbf{Ret}_\lambda)\supseteq\mathbf{Mono}$, take a monomorphism $m:A\to B$.
    Consider a commutative square $vr=mu$ and assume that $r:X\to Y$ is a retraction.
    Since $m$ is a monomorphism, a diagonal filler is unique if it exists.
    Thus, it suffices to prove the existence part.
    Taking a section $s:Y\to X$ of $r$, we have $mus=vrs=v$ in the following diagram:
    \begin{equation*}
        \begin{tikzcd}
            X\arrow[dd,"r"']\arrow[rr,"u"] & & A\arrow[dd,"m"] \\
            & X\arrow[d,"r"]\arrow[ru,"u"] & \\
            Y\arrow[ru,"s"]\arrow[r,equal] & Y\arrow[r,"v"'] & B
        \end{tikzcd}
    \end{equation*}
    Since $m$ is a monomorphism, the upper triangle of the above also commutes, which proves $us$ is a unique diagonal filler.

    We next prove $\rorth(\mathbf{Ret}_\lambda)\subseteq\mathbf{Mono}$.
    Take a morphism $f:X\to Y$ belonging to $\rorth(\mathbf{Ret}_\lambda)$.
    To show that $f$ is a monomorphism, take the following cofork arbitrarily:
    \begin{equation*}
        \begin{tikzcd}
            A\arrow[r,shift left=2,"u"]\arrow[r,shift right=2,"v"'] & X\arrow[r,"f"] & Y
        \end{tikzcd}
    \end{equation*}
    Since $\lambda$-presentable objects form a generator for $\A$, we can assume that $A$ is $\lambda$-presentable without loss of generality.
    Considering the codiagonal $\nabla:A+A\to A$, we have the following commutative square:
    \begin{equation*}
        \begin{tikzcd}
            A+A\arrow[d,"\nabla"']\arrow[r,"{(u,v)}"] & X\arrow[d,"f"] \\
            A\arrow[r,"fu(=fv)"'] & Y
        \end{tikzcd}
    \end{equation*}
    By $\nabla\in\mathbf{Ret}_\lambda$, the above square has a unique diagonal filler.
    This implies $u=v$, which shows that $f$ is a monomorphism.
\end{proof}

\begin{definition}
    Let $\bM$ be a class of monomorphisms.
    A morphism $f:X\to Y$ is \emph{$\bM$-extremal} if $f$ factors through no proper $\bM$-subobject of $Y$, i.e., 
    if $f$ has a factorization $f=mg$ with $m\in\bM$, then $m$ is an isomorphism.
\end{definition}

\begin{lemma}\label{lem:M-extremal_M-strong}
    Let $\C$ be a category with pullbacks. Let $\bM$ be a class of monomorphisms in $\C$ which is closed under pullbacks, i.e., 
    for every pullback square
    \begin{equation*}
        \begin{tikzcd}
            \cdot\arrow[r]\arrow[d,"m'"']\arrow[rd,pos=0.1,phantom,"\lrcorner"] & \cdot\arrow[d,"m"] \\
            \cdot\arrow[r] & \cdot
        \end{tikzcd}\incat{\C},
    \end{equation*}
    $m\in\bM$ implies $m'\in\bM$.
    Then, the class $\lorth(\bM)$ coincides with the class of all $\bM$-extremal morphisms in $\C$.
\end{lemma}
\begin{proof}
    The proof is straightforward.
\end{proof}

The following definition is due to \cite{hebert2004algebraically}.

\begin{definition}
    An orthogonal factorization system $(\bE,\bM)$ is \emph{$\lambda$-presentable} if $\rorth(\bE_\lambda)\subseteq\bM$.
\end{definition}

\begin{theorem}\label{thm:PPFS}
    Let $\A$ be a locally $\lambda$-presentable category.
    Let $\Lambda$ be a class of epimorphisms between $\lambda$-presentable objects in $\A$.
    Consider the following two classes of morphisms in $\A$.
    \begin{itemize}
        \item $\bM$: the class of all monomorphisms in $\A$ belonging to $\rorth(\Lambda)$,
        \item $\bE$: the class of all $\bM$-extremal morphisms in $\A$.
    \end{itemize}
    Then $(\bE,\bM)$ is a $\lambda$-presentable proper factorization system.
    Conversely, every $\lambda$-presentable proper factorization system in $\A$ is constructed as above.
\end{theorem}
\begin{proof}
    Define $\Lambda^*:=\mathbf{Ret}_\lambda \cup\Lambda$.
    Since $\Lambda^*$ is essentially small, \cref{thm:OFS_in_LPcategory} shows that
    \[
    (\,\lorth\rorth(\Lambda^*),\rorth(\Lambda^*) \,)
    \]
    is a $\lambda$-presentable orthogonal factorization system in $\A$.
    \cref{lem:mono_rlp_wrt_retracts} implies $\rorth(\Lambda^*)=\bM$, and \cref{lem:M-extremal_M-strong} implies $\lorth\rorth(\Lambda^*)=\bE$.

    Take a morphism $f:X\to Y$ belonging to $\lorth\rorth(\Lambda^*)=\bE$.
    To show that $f$ is an epimorphism, fix the following fork:
    \begin{equation*}
        \begin{tikzcd}
            X\arrow[r,"f"] & Y\arrow[r,shift left=2,"u"]\arrow[r,shift right=2,"v"'] & Z
        \end{tikzcd}
    \end{equation*}
    Considering the diagonal $\Delta:Z\to Z\times Z$, we have $e\perp\Delta$ for every epimorphism $e$, which implies $\Delta\in\rorth(\Lambda^*)$.
    Thus $f\perp\Delta$ holds, and the following square has a unique diagonal filler:
    \begin{equation*}
        \begin{tikzcd}
            X\arrow[d,"f"']\arrow[r,"uf\,(=vf)"] & Z\arrow[d,"\Delta"] \\
            Y\arrow[r,"{(u,v)}"'] & Z\times Z
        \end{tikzcd}
    \end{equation*}
    This implies $u=v$, thus $f$ is an epimorphism.
    Now we have proved that $(\bE,\bM)$ is a $\lambda$-presentable proper factorization system in $\A$.
\end{proof}

The proof of the following is essentially the same as Theorem 2.10 of \cite{adamek2009orthogonal}.

\begin{theorem}\label{thm:co-intersection}
    Let $(\bE,\bM)$ be a $\lambda$-presentable proper factorization system in a locally $\lambda$-presentable category $\A$.
    Then every morphism $f:A\to X$ in $\bE$ with $\lambda$-presentable domain $A$ is a $\lambda$-filtered colimit (in $A/\A$) of morphisms in $\bE$ with $\lambda$-presentable codomain.
\end{theorem}
\begin{proof}
    Let $e_I:A\to Y_I\,(I\in\bI)$ be the family of all morphisms belonging to $\bE$ which are factored through by $f$ and whose codomain is $\lambda$-presentable.
    Since $\bI$ yields a $\lambda$-filtered essentially small category, 
    there is a colimit $(e,Y)$ of $(e_I,Y_I)_{I\in\bI}$ in the coslice category $A/\A$.
    By the universality of the colimit $(e,Y)$, $f$ has a factorization $f=ge$ as follows:\vspace{-1ex}
    \begin{equation*}
        \begin{tikzcd}
            A\arrow[d,"e_I"']\arrow[rd,"e"description]\arrow[r,"f"] & X \\
            Y_I\arrow[r,"\kappa_I"'] & Y\arrow[u,dashed,"g"']
        \end{tikzcd}\incat{\A},\vspace{-1ex}
    \end{equation*}
    where $\kappa_I$ is the coprojection of the colimit.
    Since $\bE$ coincides with the $\bM$-extremals by \cref{thm:PPFS}, $f\in\bE$ implies $g\in\bE$.

    To prove that $g$ is an isomorphism, it suffices to show $g\in\bM$.
    To show this, take the following commutative square arbitrarily:\vspace{-1ex}
    \begin{equation}\label{eq:comm_square_dg}
        \begin{tikzcd}
            B\arrow[d,"\bE\ni\,d"']\arrow[r,"u"] & Y\arrow[d,"g"] \\
            C\arrow[r,"v"'] & X
        \end{tikzcd}\incat{\A}.\vspace{-1ex}
    \end{equation}
    We have to construct a unique diagonal filler for the above \cref{eq:comm_square_dg}.
    Since $(\bE,\bM)$ is a $\lambda$-presentable proper factorization system, the uniqueness always holds and we can assume that $B$ and $C$ are $\lambda$-presentable.
    Now $u$ has a factorization $u=\kappa_I u'$ for some $I\in\bI$ because $B$ is $\lambda$-presentable.
    Take a pushout $Z$ of $d$ and $u'$, and consider the following canonical morphism $h:Z\to X$:
    \begin{equation*}
        \begin{tikzcd}
            B\arrow[rd,pos=0.95,bend right=20,phantom,"\ulcorner"]\arrow[rr,shift left=4,"u"]\arrow[dd,"d"']\arrow[r,"u'"'] & Y_I\arrow[r,"\kappa_I"']\arrow[d,"\rho"] & Y\arrow[dd,"g"] \\
            & Z\arrow[rd,dashed,"h"] & \\
            C\arrow[ru]\arrow[rr,"v"'] & & X
        \end{tikzcd}\incat{\A}.
    \end{equation*}
    Since $B$, $C$, and $Y_I$ are $\lambda$-presentable, $Z$ is also $\lambda$-presentable.
    Now $d\in\bE$ implies $\rho\in\bE$.
    Thus, 
    \[
    A\longarr{e_I}Y_I\longarr{\rho}Z
    \]
    is a morphism belonging to $\bE$ which is factored through by $f$.
    Therefore, there exists $J\in\bI$ satisfying $Z=Y_J$ and $\rho e_I=e_J$.
    Then the following diagram commutes:
    \begin{equation*}
        \begin{tikzcd}
            B\arrow[rd,pos=0.95,bend right=20,phantom,"\ulcorner"]\arrow[rr,shift left=4,"u"]\arrow[dd,"d"']\arrow[r,"u'"'] & Y_I\arrow[r,"\kappa_I"']\arrow[d,"\rho"] & Y \\
            & Y_J\arrow[ur,"\kappa_J"'] & \\
            C\arrow[ru] & &
        \end{tikzcd}\incat{\A}.
    \end{equation*}
    Since $d$ is an epimorphism, we have constructed a diagonal filler for the square \cref{eq:comm_square_dg}.
\end{proof}

\subsection{Closed monomorphisms}

\begin{remark}
    Let $\bT$ be a partial Horn theory over an $S$-sorted signature $\Sigma$.
    Then, for every morphism $h:A\to B$ in $\bT\-\PMod$, the following are equivalent:
    \begin{enumerate}
        \item
        $h$ is a monomorphism in $\bT\-\PMod$,
        \item
        $h_s:A_s\to B_s$ is injective for every sort $s\in S$.
    \end{enumerate}
    Thus, a subobject in $\bT\-\PMod$ is just a submodel.
\end{remark}

\begin{definition}
    Let $\bT$ be a partial Horn theory over an $S$-sorted signature $\Sigma$.
    \begin{enumerate}
        \item
        A monomorphism $A\hookrightarrow B$ in $\bT\-\PMod$ is called \emph{$\bT$-closed} (or \emph{$\Sigma$-closed}) if the following diagrams form pullback squares for any $f,R\in\Sigma$.
        \begin{equation*}
        \begin{tikzcd}
            \mathrm{Dom}(\intpn{f}{A})\arrow[d,hook']\arrow[r,hook]\arrow[rd,pos=0.1,phantom,"\lrcorner"] &[-10pt] A_{s_1}\times\dots\times A_{s_n}\arrow[d,hook'] \\
            \mathrm{Dom}(\intpn{f}{B})\arrow[r,hook] & B_{s_1}\times\dots\times B_{s_n}
        \end{tikzcd}
        \quad\quad
        \begin{tikzcd}
            \intpn{R}{A}\arrow[d,hook']\arrow[r,hook]\arrow[r,hook]\arrow[rd,pos=0.1,phantom,"\lrcorner"] &[-10pt] A_{s_1}\times\dots\times A_{s_n}\arrow[d,hook'] \\
            {\intpn{R}{B}} \arrow[r,hook] & B_{s_1}\times\dots\times B_{s_n}
        \end{tikzcd}
        \end{equation*}
        \item
        A morphism $h:A\to B$ in $\bT\-\PMod$ is called \emph{$\bT$-dense} (or \emph{$\Sigma$-dense}) if $h$ factors through no $\bT$-closed proper subobject of $B$.
    \end{enumerate}
\end{definition}

$\bT$-closed monomorphisms are the so-called embeddings in model theory and play an important role in our generalized Birkhoff's theorem.
Our term ``$\bT$-closed'' is based on \cite{burmeister2002lecture}.

\begin{remark}\label{rem:explaination_closedmono_dense}
    Let $\bT$ be a partial Horn theory over an $S$-sorted signature $\Sigma$.
    \begin{enumerate}
        \item
        $\bT$-closedness of a submodel $A\subseteq B$ in $\bT\-\PMod$ is equivalent to saying that given a family $\tup{a}$ of elements of $A$, if $\intpn{f}{B}(\tup{a})$ is defined for a function symbol $f$, then $\intpn{f}{A}(\tup{a})$ is also defined, and if $\tup{a}\in\intpn{R}{B}$ holds for a relation symbol $R$, then $\tup{a}\in\intpn{R}{A}$ also holds.
        \item
        Let $C\subseteq B$ be a subobject of $B\in\bT\-\PMod$ in $\Set^S$.
        Denote by $A$ the $S$-sorted set of all elements of $B$ which can be written as $\intpn{\tup{x}.\tau}{B}(\tup{c})$ by a family $\tup{c}$ of elements of $C$ and a term $\tau$ over $\Sigma$.
        Then the $\Sigma$-structure of $B$ induces a $\Sigma$-structure on $A$, which makes $A$ the smallest $\bT$-closed submodel of $B$ containing $C$.
        This $A$ is called the \emph{$\bT$-closed submodel of $B$ generated by $C$}.
        \item\label{rem:explaination_closedmono_dense-3}
        A morphism $h:A\to B$ in $\bT\-\PMod$ is $\bT$-dense if and only if the $\bT$-closed submodel generated by the image of $h$ coincides with $B$.
    \end{enumerate}
\end{remark}

\begin{example}\label{eg:PHT_for_monoids}\quad
    \begin{enumerate}
        \item
        Let us define an ordinary partial Horn theory $\bT_\mon$ (over $\Sigma_\mon$) for monoids as follows:
        \begin{gather*}
            S:=\{*\},\quad \Sigma_\mon:=\{ e:1\to *,\quad \cdot:*\times *\to * \},
            \\
            \bT_\mon:=\left\{
            \begin{gathered}
                \top\seq{}e\defined,\quad \top\seq{x,y}x\cdot y\defined,\\
                \top\seq{x,y,z}(x\cdot y)\cdot z=x\cdot(y\cdot z),\\
                \top\seq{x}x\cdot e=x \wedge e\cdot x=x
            \end{gathered}
            \right\}.
        \end{gather*}
        We have $\bT_\mon\-\PMod\cong\Mon$, where $\Mon$ is the category of monoids.
        Then a $\bT_\mon$-closed subobject is just a submonoid, and a $\bT_\mon$-dense morphism is just a surjective homomorphism.
        \item
        Let us define another partial Horn theory $\bT'_\mon$ (over $\Sigma'_\mon$) for monoids as follows:
        \begin{gather*}
            S:=\{*\},\quad \Sigma'_\mon:=\Sigma_\mon+\{ \bullet^{-1}:*\to * \},
            \\
            \bT'_\mon:=\bT_\mon+\left\{
            \begin{gathered}
                x^{-1}\defined\seq{x}x^{-1}\cdot x=e \wedge x\cdot x^{-1}=e,\\
                y\cdot x=e \wedge x\cdot y=e\seq{x,y}x^{-1}=y
            \end{gathered}
            \right\}.
        \end{gather*}
        A $\bT'_\mon$-model is just a monoid with the partial inverse function, and we also have $\bT'_\mon\-\PMod\cong\Mon$.
        Then a submonoid $\bN\subseteq\bZ$ is not $\bT'_\mon$-closed even though it is $\bT_\mon$-closed.
        Therefore closedness of monomorphisms depends on $\bT$.
    \end{enumerate}
\end{example}

\begin{theorem}\label{thm:dense_closedmono_factorization}
    Let $\bT$ be a partial Horn theory over an $S$-sorted signature $\Sigma$.
    Then a pair of the class of all $\bT$-dense morphisms and the class of all $\bT$-closed monomorphisms is a finitely presentable proper factorization system in $\bT\-\PMod$.
\end{theorem}
\begin{proof}
    Denote by $\Lambda$ the class of morphisms in $\bT\-\PMod$ consisting of the following:
    \begin{itemize}
        \item
        A morphism $\repn{\tup{x}.\top}_\bT\longarr{\repn{\tup{x}}_\bT}\repn{\tup{x}.f(\tup{x})\defined}_\bT$ for each function symbol $f\in\Sigma$,
        \item
        A morphism $\repn{\tup{x}.\top}_\bT\longarr{\repn{\tup{x}}_\bT}\repn{\tup{x}.R(\tup{x})}_\bT$ for each relation symbol $R\in\Sigma$.
    \end{itemize}
    Then $\bT$-closedness of a monomorphism $m$ is equivalent to whether $m$ belongs to $\rorth(\Lambda)$.
    Since $\Lambda$ is a small class of epimorphisms between finitely presentable objects, the statement follows from \cref{thm:PPFS}.
\end{proof}

\begin{lemma}\label{lem:dense_between_fp}
    Let $\bT$ be a partial Horn theory over an $S$-sorted signature $\Sigma$.
    Then every $\bT$-dense morphism between finitely presentable objects has the following expression:
    \begin{equation*}
        \repn{\tup{x}.\phi}_\bT \longarr{\repn{\tup{x}}_\bT} \repn{\tup{x}.\psi}_\bT
    \end{equation*}
\end{lemma}
\begin{proof}
    By \cref{cor:morphism_between_fp}, a morphism between finitely presentable objects has the following expression:
    \begin{equation}\label{eq:dense_between_fp_tau}
        \repn{\tup{x}.\phi}_\bT \longarr{\repn{\tup{\tau}}_\bT} \repn{\tup{y}.\psi}_\bT,
    \end{equation}
    where $\tup{x}.\phi$ and $\tup{x}.\psi$ are Horn formulas-in-context over $\Sigma$ with $\tup{x}=(x_1,\dots,x_n)$ and $\tup{y}=(y_1,\dots,y_m)$.
    Suppose the morphism \cref{eq:dense_between_fp_tau} is $\bT$-dense.
    By \cref{rem:explaination_closedmono_dense}\ref{rem:explaination_closedmono_dense-3}, 
    For each $1\le j\le m$, we can take a term-in-context $\tup{x}.\sigma_j$ satisfying $[y_j]_\bT=[\sigma_j(\tup{\tau}/\tup{x})]_\bT$ in $\repn{\tup{y}.\psi}_\bT$, 
    i.e., the following is a PHL-theorem of $\bT$:
    \begin{equation*}
        \psi \seq{\tup{y}} y_j=\sigma_j(\tup{\tau}/\tup{x}).
    \end{equation*}
    Denote by $\chi$ the following Horn formula:
    \begin{equation*}
        \chi:=\quad \psi(\tup{\sigma}/\tup{y})\wedge\bigwedge_{1\le i\le n}x_i=\tau_i(\tup{\sigma}/\tup{y})\wedge\bigwedge_{1\le j\le m}\sigma_j\defined.
    \end{equation*}

    Now it is easy to check that the following are PHL-theorems of $\bT$:
    \begin{gather}
        \psi \seq{\tup{y}} \bigwedge_{1\le i\le n}\tau_i\defined,
        \quad\quad
        \psi \seq{\tup{y}} \chi(\tup{\tau}/\tup{x});
        \label{eq:welldefinedness_tau}
        \\
        \chi \seq{\tup{x}} \bigwedge_{1\le j\le m}\sigma_j\defined,
        \quad\quad
        \chi \seq{\tup{x}} \psi(\tup{\sigma}/\tup{y});
        \label{eq:welldefinedness_sigma}
        \\
        \psi\seq{\tup{y}}\bigwedge_{1\le j\le m}y_j=\sigma_j(\tup{\tau}/\tup{x}),
        \quad\quad
        \chi\seq{\tup{x}}\bigwedge_{1\le i\le n}x_i=\tau_i(\tup{\sigma}/\tup{y});
        \label{eq:isomorphism_sigma_tau}
        \\
        \chi\seq{\tup{x}}\phi.
        \label{eq:phi_implies_chi}
    \end{gather}
    \cref{eq:welldefinedness_sigma,eq:welldefinedness_tau} being PHL-theorems of $\bT$ implies the well-definedness of the following morphisms:
    \begin{equation}\label{eq:mor_sigma_tau}
        \begin{tikzcd}
            \repn{\tup{x}.\chi}_\bT & \repn{\tup{y}.\psi}_\bT
            \arrow[from=1-2,to=1-1,shift left=2,"\repn{\tup{\sigma}}_\bT"]
            \arrow[from=1-1,to=1-2,shift left=2,"\repn{\tup{\tau}}_\bT"]
        \end{tikzcd}
    \end{equation}
    Since \cref{eq:isomorphism_sigma_tau} is a PHL-theorem of $\bT$, two morphisms in \cref{eq:mor_sigma_tau} are inverse of each other.
    \cref{eq:phi_implies_chi} yields a morphism $\repn{\tup{x}.\phi}_\bT\longarr{\repn{\tup{x}}_\bT}\repn{\tup{x}.\chi}_\bT$, and we have the following commutative diagram:
    \begin{equation*}
        \begin{tikzcd}
            &[-20pt] \repn{\tup{x}.\phi}_\bT &[-20pt] \\
            \repn{\tup{x}.\chi}_\bT & & \repn{\tup{y}.\psi}_\bT
            \arrow[from=1-2,to=2-1,"\repn{\tup{x}}_\bT"']
            \arrow[from=1-2,to=2-3,"\repn{\tup{\tau}}_\bT"]
            \arrow[from=2-1,to=2-3,"\repn{\tup{\tau}}_\bT"',"\cong"]
        \end{tikzcd}
    \end{equation*}
    This completes the proof.
\end{proof}

\subsection{Birkhoff's variety theorem for relative algebraic theories}
\begin{lemma}[\cite{baron1969reflectors}]\label{lem:epi-reflective}
    Let $\A\subseteq\B$ be a replete full subcategory of a locally small category $\B$ with products.
    Let $\bE$ be a class of epimorphisms in $\B$.
    Assume that for every $B\in\B$, the class
    \[
    B/_\bE\A:=\{ q:B\to\cod q \mid q\in\bE\text{ and }\cod q\in\A \}
    \]
    is essentially small.
    Then the following are equivalent:
    \begin{enumerate}
        \item\label{lem:epi-reflective-1}
        $\A\subseteq\B$ is $\bE$-reflective, i.e., 
        the inclusion $\A\hookrightarrow\B$ has a left adjoint and all components of its unit belong to $\bE$.
        \item\label{lem:epi-reflective-2}
        \begin{itemize}
            \item
            $\A\subseteq\B$ is closed under products, and
            \item
            every morphism $f:B\to A$ in $\B$ with $A\in\A$ factors through a morphism $e:B\to A'$ with $A'\in\A$ belonging to $\bE$.
            \begin{equation*}
                \begin{tikzcd}
                    B\arrow[rr,"f"]\arrow[rd,"\exists e"',dashed] && A \\
                    & \exists A'\arrow[ur,"\exists"',dashed] &
                \end{tikzcd}
            \end{equation*}
        \end{itemize}
    \end{enumerate}
\end{lemma}
\begin{proof}
    {[\ref{lem:epi-reflective-1}$\implies$\ref{lem:epi-reflective-2}]}
    Immediate.

    {[\ref{lem:epi-reflective-2}$\implies$\ref{lem:epi-reflective-1}]}
    Fix an object $B\in\B$.
    Since the class $B/_\bE\A$ is essentially small, we can take a small skeleton $\{ B\arr{q_i}A_i \}_{i\in I}$ of $B/_\bE\A$.
    Taking a canonical morphism $(q_i)_i:B\to\prod_{i}A_i$, we have $\prod_i A_i\in\A$ by the assumption.
    Then the morphism $(q_i)_i$ factors through a morphism $\eta\in\bE$ with its codomain $rB\in\A$ as follows:
    \begin{equation*}
        \begin{tikzcd}
            B\arrow[rr,"(q_i)_i"]\arrow[rd,"\eta"',dashed] && \prod_{i\in I} A_i \\
            & rB\arrow[ur,"\exists"',dashed] &
        \end{tikzcd}
    \end{equation*}
    We proceed to show that $\eta$ is a reflection of $B$ into $\A$.
    Every morphism $f:B\to A$ with $A\in\A$ factors through $q_{i_0}$ for some $i_0\in I$.
    This provides the following commutative diagram:
    \begin{equation*}
        \begin{tikzcd}[column sep=large,row sep=large]
            B\arrow[rr,"f"]\arrow[d,"\eta"']\arrow[rd,"(q_i)_i"description]\arrow[rrd,"q_{i_0}"] & & A \\
            rB\arrow[r] & \prod_i A_i\arrow[r,"\text{projection}"'] & A_{i_0}\arrow[u,dashed]
        \end{tikzcd}
    \end{equation*}
    Therefore $f$ factors through $\eta$, and its factorization is unique since $\eta\in\bE$ is an epimorphism.
\end{proof}

\begin{corollary}\label{cor:E-reflective}
    Let $\A$ be a locally presentable category with a proper factorization system $(\bE,\bM)$.
    For any replete full subcategory $\E\subseteq\A$, the following are equivalent:
    \begin{enumerate}
        \item
        $\E\subseteq\A$ is $\bE$-reflective.
        \item
        $\E\subseteq\A$ is closed under products and $\bM$-subobjects.
    \end{enumerate}
\end{corollary}
\begin{proof}
    Since every locally presentable category is co-wellpowered (see \cite[1.58 Theorem]{adamek1994locally}), this follows from \cref{lem:epi-reflective}.
\end{proof}

\begin{definition}
    Let $\rho:(S,\Sigma,\bS)\to (S',\Sigma',\bT)$ be a theory morphism between partial Horn theories.
    A \emph{$\rho$-relative judgment} is a Horn sequent $\phi^\rho\seq{\tup{x}^\rho}\psi$, where $\tup{x}.\phi$ is a Horn formula-in-context over $\Sigma$ and $\tup{x}^\rho.\psi$ is a Horn formula-in-context over $\Sigma'$.
\end{definition}

We can now formulate our main result.

\begin{theorem}\label{thm:birkhoff_partialHorn}
    Let $\rho:\bS\to\bT$ be a theory morphism between partial Horn theories.
    Then, for every replete full subcategory $\E\subseteq\bT\-\PMod$, the following are equivalent:
    \begin{enumerate}
        \item\label{thm:birkhoff_partialHorn-1}
        $\E$ is definable by $\rho$-relative judgments, i.e., 
        there exists a set $\bT'$ of $\rho$-relative judgments satisfying $\E=(\bT+\bT')\-\PMod$.
        \item\label{thm:birkhoff_partialHorn-2}
        $\E\subseteq\bT\-\PMod$ is closed under products, $\bT$-closed subobjects, $U^\rho$-retracts, and filtered colimits.
    \end{enumerate}
\end{theorem}
\begin{proof}
    {[\ref{thm:birkhoff_partialHorn-1}$\implies$\ref{thm:birkhoff_partialHorn-2}]}
    Consider the following adjunction:
    \begin{equation}\label{eq:adj_F^rhoU^rho}
    \begin{tikzcd}[column sep=large, row sep=large]
        \bS\-\PMod\arrow[r,"F^\rho",shift left=7pt]\arrow[r,"\perp"pos=0.5,phantom] &[10pt]\bT\-\PMod\arrow[l,"U^\rho",shift left=7pt]
    \end{tikzcd}
    \end{equation}
    Note that $U^\rho$ preserves filtered colimits.
    Let $\bT'=\{\,\phi^\rho_i\seq{\tup{x}^\rho_i}\psi_i\,\}_{i\in I}$.
    Then, for every $M\in\bT\-\PMod$, we obtain the following natural isomorphisms:
    \begin{multline*}
        \bT\-\PMod(F^\rho\repn{\tup{x}_i.\phi_i}_\bS , M)
        \cong
        \bS\-\PMod(\repn{\tup{x}_i.\phi_i}_\bS , U^\rho M) \\
        \cong
        \intpn{\tup{x}_i.\phi_i}{U^\rho M}
        =
        \intpn{\tup{x}^\rho_i.\phi^\rho_i}{M}
        \cong
        \bT\-\PMod(\repn{\tup{x}^\rho_i.\phi^\rho_i}_\bT , M).
    \end{multline*}
    This admits us to consider $F^\rho\repn{\tup{x}_i.\phi_i}_\bS=\repn{\tup{x}^\rho_i.\phi^\rho_i}_\bT$.
    By \cref{prop:validity_for_PHL}, $(\bT+\bT')\-\PMod\subseteq\bT\-\PMod$ is an orthogonality class with respect to the following class of morphisms:
    \begin{equation*}
        \Lambda:=\{~F^\rho\repn{\tup{x}_i.\phi_i}_\bS \arr{\repn{\tup{x}^\rho_i}_\bT} \repn{\tup{x}^\rho_i.\phi^\rho_i\wedge \psi_i}_\bT~\}_{i\in I}.
    \end{equation*}
    Orthogonality classes are in general closed under products.
    A similar argument as in the proof of \cref{lem:closed_under_localret_filteredcolim} shows that $(\bT+\bT')\-\PMod\subseteq\bT\-\PMod$ is closed under $U^\rho$-retracts and filtered colimits.
    In the following, we show that $(\bT+\bT')\-\PMod\subseteq\bT\-\PMod$ is closed under $\bT$-closed subobjects.
    Let $m:M\hookrightarrow N$ be a $\bT$-closed monomorphism in $\bT\-\PMod$, and assume $N\in\orth{\Lambda}=(\bT+\bT')\-\PMod$.
    Take a morphism $f:F^\rho\repn{\tup{x}_i.\phi_i}_\bS\to M$ arbitrarily.
    By $N\in\orth{\Lambda}$, there exists a unique morphism $g:\repn{\tup{x}^\rho_i.\phi^\rho_i\wedge \psi_i}_\bT\to N$ which makes the following commutes.
    \begin{equation}\label{eq:comm_square_xm}
        \begin{tikzcd}[column sep=huge, row sep=huge]
            F^\rho\repn{\tup{x}_i.\phi_i}_\bS & M \\
            \repn{\tup{x}^\rho_i.\phi^\rho_i\wedge \psi_i}_\bT & N
            \arrow[from=1-1,to=2-1,"\repn{\tup{x}^\rho_i}_\bT"']
            \arrow[from=1-2,to=2-2,hook',"m"]
            \arrow[from=1-1,to=1-2,"f"]
            \arrow[from=2-1,to=2-2,dashed,"\exists !g"']
        \end{tikzcd}\incat{\bT\-\PMod}
    \end{equation}
    Since $\repn{\tup{x}^\rho_i}_\bT$ is $\bT$-dense, \cref{thm:dense_closedmono_factorization} ensures the unique existence of a diagonal filler for \cref{eq:comm_square_xm}.
    This proves that $M$ belongs to $\orth{\Lambda}=(\bT+\bT')\-\PMod$.

    {[\ref{thm:birkhoff_partialHorn-2}$\implies$\ref{thm:birkhoff_partialHorn-1}]}
    In what follows, let $\A:=\bS\-\PMod$.
    Since $\E\subseteq\bT\-\PMod$ is closed under products and $\bT$-closed subobjects, 
    \cref{cor:E-reflective} claims that the adjoint
    \begin{equation*}
    \begin{tikzcd}[column sep=large, row sep=large]
        \bT\-\PMod\arrow[r,"r",shift left=7pt]\arrow[r,"\perp"pos=0.5,phantom] &[10pt]\E\arrow[l,shift left=7pt,hook']
    \end{tikzcd}
    \end{equation*}
    exists, and the $M(\in\bT\-\PMod)$-component of its unit $e$
    \begin{equation*}
        M\longarr{e_M}rM\incat{\bT\-\PMod}
    \end{equation*}
    is $\bT$-dense.
    Consider the following class of morphisms in $\bT\-\PMod$:
    \begin{gather*}
        \Lambda:=\{\, F^\rho A\longarr{e_{F^\rho A}}rF^\rho A \,\}_{A\in\fp{\A}},
        \\
        \Lambda^*:=\{\, F^\rho X\longarr{e_{F^\rho X}}rF^\rho X \,\}_{X\in\A}.
    \end{gather*}
    Every $X\in\A$ can be presented as a filtered colimit $X=\Colim{I\in\bI}X_I$ which $X_I$ are finitely presentable.
    Since $\E\subseteq\bT\-\PMod$ is closed under filtered colimits, 
    $e_{F^\rho X}=\Colim{I\in\bI}e_{F^\rho X_I}$ is a filtered colimit in the arrow category $(\bT\-\PMod)^\to$.
    Thus, $\orth{\Lambda}=\orth{\Lambda^*}$ holds.

    Take an object $M\in\bT\-\PMod$ satisfying $M\in\orth{\Lambda^*}$.
    Let denote by $\epsilon$ the counit of the adjunction \cref{eq:adj_F^rhoU^rho}.
    By $e_{F^\rho U^\rho (M)}\in\Lambda^*$, there exists a unique morphism $p$ which makes the following commutes:
    \begin{equation*}
        \begin{tikzcd}[column sep=huge, row sep=huge]
            F^\rho U^\rho(M)\arrow[d,"e_{F^\rho U^\rho(M)}"']\arrow[r,"\epsilon_M"] & M \\
            rF^\rho U^\rho(M)\arrow[ur,dashed,"\exists !p"'] &
        \end{tikzcd}\incat{\bT\-\PMod}
    \end{equation*}
    Since $U^\rho\epsilon_M$ is a retraction, $U^\rho p$ is also a retraction, which implies $p$ is a $U^\rho$-retraction.
    Since $\E\subseteq\bT\-\PMod$ is closed under $U^\rho$-retracts, we have $M\in\E$.
    By the above argument, we have $\orth{\Lambda}=\E$.

    For each finitely presentable object $A\in\fp{\A}$, 
    choose a Horn formula $\tup{x}_A.\phi_A$ satisfying $A\cong\repn{\tup{x}_A.\phi_A}_\bS$.
    By \cref{thm:co-intersection}, $e_{F^\rho A}\in{F^\rho A}/\bT\-\PMod$ can be presented as the following filtered colimit of $\bT$-dense morphisms $e_{F^\rho A}^{(j)}$ whose codomain is finitely presentable:
   \[
   e_{F^\rho A}=\Colim{j}e_{F^\rho A}^{(j)} \incat{{F^\rho A}/\bT\-\PMod}.
   \]
    By \cref{lem:dense_between_fp}, each $e_{F^\rho A}^{(j)}$ has the following presentation:
    \begin{equation*}
        F^\rho\repn{\tup{x}_A.\phi_A}_\bS \longarr{\repn{\tup{x}^\rho_A}_\bT} \repn{\tup{x}^\rho_A.\psi_A^{(j)}}_\bT.
    \end{equation*}
    Orthogonality to all $e_{F^\rho A}^{(j)}$ is equivalent to orthogonality to $e_{F^\rho A}$.
    Thus, taking
    \begin{equation*}
        \bT':=\{\, \phi^\rho_A\seq{\tup{x}^\rho_A}\psi_A^{(j)} \,\}_{A,j},
    \end{equation*}
    we have $(\bT+\bT')\-\PMod=\orth{\Lambda}=\E$.
\end{proof}

The above theorem has several useful corollaries.
Taking $\rho$ as the trivial one $\rho:(S,\varnothing,\varnothing)\to (S,\Sigma,\bT)$, we obtain the first corollary:
\begin{corollary}
    Let $\bT$ be a partial Horn theory over $\Sigma$.
    Then, for every replete full subcategory $\E\subseteq\bT\-\PMod$, the following are equivalent:
    \begin{enumerate}
        \item
        $\E$ is definable by Horn formulas, i.e., 
        there exists a set $E$ of Horn formulas satisfying $\E=(\bT+\bT')\-\PMod$, where $\bT':=\{ \top\seq{\tup{x}}\phi\}_{\tup{x}.\phi\in E}.$
        \item
        $\E\subseteq\bT\-\PMod$ is closed under products, $\Sigma$-closed subobjects, surjections, and filtered colimits.
    \end{enumerate}
\end{corollary}

Taking $\rho$ as the identity $\bT\to\bT$, we obtain the second corollary:
\begin{corollary}
    Let $\bT$ be a partial Horn theory over $\Sigma$.
    Then, for every replete full subcategory $\E\subseteq\bT\-\PMod$, the following are equivalent:
    \begin{enumerate}
        \item
        $\E$ is definable by Horn sequents, i.e., 
        there exists a set $\bT'$ of Horn sequents satisfying $\E=(\bT+\bT')\-\PMod$.
        \item
        $\E\subseteq\bT\-\PMod$ is closed under products, $\Sigma$-closed subobjects, and filtered colimits.
    \end{enumerate}
\end{corollary}
\begin{proof}
    Since every split monomorphism is $\Sigma$-closed, being closed under $\Sigma$-closed subobjects implies being closed under retracts.
    Thus, this is a corollary of \cref{thm:birkhoff_partialHorn}.
\end{proof}

\begin{example}
    Given a replete full subcategory $\E$ of a locally finitely presentable category $\A$, 
    whether $\E\subseteq\A$ is definable by Horn sequents depends on the choice of a partial Horn theory $\bT$ for $\A$.
    For example, let $\A:=\Mon$ be the category of monoids and let $\E:=\Grp$ be the category of groups, and consider the partial Horn theories $\bT_\mon$ and $\bT'_\mon$ as in \cref{eg:PHT_for_monoids}.
    Then $\Grp\subseteq\bT'_\mon\-\PMod$ is definable by the following Horn sequent:
    \begin{gather*}
        \top\seq{x} x^{-1}\defined.
    \end{gather*}
    On the other hand, $\Grp\subseteq\bT_\mon\-\PMod$ cannot be definable by Horn sequents, because $\bN$ is a $\bT_\mon$-closed subobject of $\bZ$ and not a group even though $\bZ$ is a group.
\end{example}

Taking $\rho$ as a ``relative algebraic theory,'' we have the following theorem.
This theorem is a generalization of Birkhoff's variety theorem from classical algebraic theories to our relative algebraic theories.

\begin{theorem}\label{thm:birkhoff_for_rat}
    Let $\bS$ be a partial Horn theory over an $S$-sorted signature $\Sigma$.
    Let $(\Omega,E)$ be an $\bS$-relative algebraic theory with the forgetful functor $U:\Alg(\Omega,E)\to\bS\-\PMod$.
    Then, for every replete full subcategory $\E\subseteq\Alg(\Omega,E)$, the following are equivalent:
    \begin{enumerate}
        \item
        $\E$ is definable by $\bS$-relative judgments, i.e., 
        there exists a set $E'$ of $\bS$-relative judgments satisfying $\E=\Alg(\Omega,E+E')$.
        \item
        $\E\subseteq\Alg(\Omega,E)$ is closed under products, $\Sigma$-closed subobjects, $U$-retracts, and filtered colimits.
    \end{enumerate}
\end{theorem}
\begin{proof}
    Consider the partial Horn theory $\pht{\Omega}{E}$ as in \cref{eq:PHT_for_OmegaE} and the extension $\rho:(S,\Sigma,\bS)\to (S,\Sigma+\Omega,\pht{\Omega}{E})$.
    Then the forgetful functor $U$ coincides with $U^\rho$.
    Since the domain of every operator in $\Omega$ is determined by a formula over $\Sigma$, $(\Sigma+\Omega)$-closedness is equivalent to $\Sigma$-closedness.
    Thus we have proved this theorem by \cref{thm:birkhoff_partialHorn}.
\end{proof}

\section{Finitary monads on locally finitely presentable categories}\label{section5}
\subsection{From finitary monads to relative algebraic theories}
We now prove that every finitary monad on a locally finitely presentable category arises from a relative algebraic theory.
This is a converse of the result in \cref{subsection3.2}.
In this subsection, we fix an $S$-sorted signature $\Sigma$ and a partial Horn theory $\bS$ over $\Sigma$.


\begin{definition}
    Let $T$ be a finitary monad on $\bS\-\PMod$.
    We now define an $\bS$-relative signature $\Omega_T$ for the finitary monad $T$:
    For each Horn formula-in-context $\tup{x}.\phi$ over $\Sigma$ and each sort $s\in S$, we have a set $(T\repn{\tup{x}.\phi}_\bS)_s$.
    We regard each element $\omega\in(T\repn{\tup{x}.\phi}_\bS)_s$ as an operator with $\ar(\omega):=\tup{x}.\phi$ and $\type(\omega):=s$, and define $\Omega_T$ to be the set of all such operators:
    \begin{equation*}
        \Omega_T:=\coprod_{\substack{\text{Horn formula }\tup{x}.\phi\text{ over }\Sigma  \\  \text{sort }s\in S}} (T\repn{\tup{x}.\phi}_\bS)_s
    \end{equation*}
\end{definition}

\begin{definition}
    Let $T$ be a finitary monad on $\A:=\bS\-\PMod$.
    We now define a natural transformation $\alpha_T:H_{\Omega_T}\Rightarrow T$, where $H_{\Omega_T}$ is the finitary endofunctor for $\Omega_T$ as in \cref{def:endofunc_for_signature}.
    In what follows, we regard each $\omega\in\Omega_T$ as a morphism $\repn{x{:}\type(\omega).\top}_\bS\to T\repn{\ar(\omega)}_\bS$ in $\A$ by \cref{prop:repn_obj_represents_intpn}.
    First, for each $A\in\A$ and $\omega\in\Omega_T$, we have the following map:
    \begin{equation*}
        \intpn{\ar(\omega)}{A} \cong \A(\repn{\ar(\omega)}_\bS, A) \ni f\quad\mapsto\quad (Tf)\circ\omega\in \A(\repn{x{:}\type(\omega).\top}_\bS, TA)
    \end{equation*}
    Then, the following bijective correspondence yields a natural transformation $\alpha:H_{\Omega_T}\Rightarrow T$:
    \begin{center}
    \renewcommand{\arraystretch}{1.4}
    \begin{tabular}{c}
        $\intpn{\ar(\omega)}{A}\longarr{}\A( \repn{x{:}\type(\omega).\top}_\bS , TA )\incat{\Set}\quad(\omega\in\Omega_T)$
        \\
        \hline\hline
        $\intpn{\ar(\omega)}{A}\bullet\repn{x{:}\type(\omega).\top}_\bS \longarr{} TA\incat{\A}\quad(\omega\in\Omega_T)$
        \\
        \hline\hline
        $H_{\Omega_T}(A)\longarr{\alpha_{T,A}}TA\incat{\A}$
    \end{tabular}
    \renewcommand{\arraystretch}{1}
    \end{center}
\end{definition}

\begin{lemma}\label{lem:alpha_is_dense}
    Let $T$ be a finitary monad on $\A:=\bS\-\PMod$.
    Then, every component of the natural transformation $\alpha_T$ is $\Sigma$-dense and particularly epic.
\end{lemma}
\begin{proof}
    We have to show that $\alpha_{T,A}:H_{\Omega_T}(A)\to TA$ is $\Sigma$-dense for every $A\in\A$.
    We first show the case where $A$ is finitely presentable.
    By \cref{thm:repn_enumerates_fpobjects}, it suffices to show the case $A=\repn{\tup{x}.\phi}_\bS$, where $\tup{x}.\phi$ is a Horn formula-in-context over $\Sigma$.
    Let $\omega\in (T\repn{\tup{x}.\phi}_\bS)_s$ be an arbitrary element.
    Then, $\omega$ is an operator with $\ar(\omega)=\tup{x}.\phi$ and $\type(\omega)=s$.
    Let $\iota:\repn{x{:}s.\top}_\bS \to \coprod_{\omega'}\A(\repn{\ar(\omega')}_\bS , A)\bullet\repn{x{:}\type(\omega').\top}_\bS \cong H_{\Omega_T}(A)$ be the coprojection for $\omega$ and $\id\in\A(\repn{\tup{x}.\phi}_\bS, A)=\A(\repn{\ar(\omega)}_\bS,A)$.
    Then, by definition of $\alpha_T$, the following commutes:
    \begin{equation*}
        \begin{tikzcd}
            \repn{x{:}s.\top}_\bS\arrow[rd,"\omega"']\arrow[r,"\iota"] & H_{\Omega_T}(A)\arrow[d,"\alpha_{T,A}"] \\
            & TA
        \end{tikzcd}\incat{\A}
    \end{equation*}
    In particular, $\omega\in (TA)_s$ lies in the image of $\alpha_{T,A}$.
    By \cref{rem:explaination_closedmono_dense}\ref{rem:explaination_closedmono_dense-3}, we now concludes that $\alpha_{T,A}$ is $\Sigma$-dense for every finitely presentable object $A\in\A$.

    We now turn to the general case.
    Given $A\in\A$, we can take a filtered colimit $A=\Colim{I\in\bI}A_I$ such that each $A_I$ is finitely presentable.
    Since $H_{\Omega_T}$ and $T$ are finitary, we have $\alpha_{T,A}=\Colim{I\in\bI}\alpha_{T,A_I}$ in the arrow category $\A^\to$.
    Since the class of all $\Sigma$-dense morphisms is a left class of an orthogonal factorization system in $\A$ by \cref{thm:dense_closedmono_factorization}, any colimit of $\Sigma$-dense morphisms is also $\Sigma$-dense.
    Thus, we concludes that $\alpha_{T,A}$ is $\Sigma$-dense.
\end{proof}

\begin{lemma}\label{lem:componentwize_epi_induce_fullyfaithful}
    Let $H$ and $K$ be endofunctors on a category $\C$ and let $\alpha:H\Rightarrow K$ be a natural transformation whose components are epimorphisms in $\C$.
    Let $\Alg\alpha:\Alg K\to\Alg H$ be the induced functor which assigns $(X,x\circ\alpha_X)$ to each object $(X,x)$.
    Then, the following hold:
    \begin{enumerate}
        \item
        $\Alg\alpha$ is fully faithful and injective on objects.
        \item
        The image of $\Alg\alpha$ is replete in $\Alg H$.
    \end{enumerate}
\end{lemma}
\begin{proof}\quad
    \begin{enumerate}
        \item
        If $(X,x\circ\alpha_X)=(Y,y\circ\alpha_Y)$, then $X=Y$ and $x=y$ since $\alpha_X=\alpha_Y$ is an epimorphism.
        Thus, $\Alg\alpha$ is injective on objects.
        To prove fully faithfulness, let $(X,x)$ and $(Y,y)$ be objects in $\Alg K$ and let $f:(X,x\circ\alpha_X)\to (Y,y\circ\alpha_Y)$ be a morphism in $\Alg H$.
        Then, in the following diagram, both the outer rectangle and the upper square commute:
        \begin{equation*}
            \begin{tikzcd}
                HX\arrow[d,"\alpha_X"']\arrow[r,"Hf"] & HY\arrow[d,"\alpha_Y"] \\
                KX\arrow[d,"x"']\arrow[r,"Kf"] & KY\arrow[d,"y"] \\
                X\arrow[r,"f"] & Y
            \end{tikzcd}\incat{\C}
        \end{equation*}
        Since $\alpha_X$ is an epimorphism, the lower square of the above diagram also commutes, which finishes the proof.
        \item
        Let $(X,x)$ be an object in $\Alg K$ and let $f:(X,x\circ\alpha_X)\to (Y,y)$ be an isomorphism in $\Alg H$.
        Then, we can easily show that $y=f\circ x\circ (Kf^{-1})\circ\alpha_Y$.
        Thus, $(Y,y)$ lies in the image of $\Alg\alpha$.\qedhere
    \end{enumerate}
\end{proof}

\begin{lemma}\label{lem:nat_square_pushout}
    Let $F,G:\C\to\D$ be functors and let $\alpha:F\Rightarrow G$ be a natural transformation.
    Let $r:X\to Y$ be a retraction in $\C$ and assume that $\alpha_X$ is an epimorphism in $\D$.
    Then, the following diagram forms a pushout square:
    \begin{equation*}
        \begin{tikzcd}
            FX\arrow[d,"\alpha_X"']\arrow[r,"Fr"] & FY\arrow[d,"\alpha_Y"] \\
            GX\arrow[r,"Gr"] & GY
        \end{tikzcd}\incat{\D}
    \end{equation*}
\end{lemma}
\begin{proof}
    Let $f:FY\to Z$ and $g:GX\to Z$ be morphisms in $\D$ such that $f\circ Fr=g\circ\alpha_X$.
    We have to construct a unique morphism $h$ which makes the following commute:
    \begin{equation}\label{eq:canonical_morphism_h}
        \begin{tikzcd}
            FX\arrow[d,"\alpha_X"']\arrow[r,"Fr"] & FY\arrow[d,"\alpha_Y"']\arrow[rdd,"f"] & \\
            GX\arrow[r,"Gr"]\arrow[rrd,"g"'] & GY\arrow[rd,dashed,"h"description] & \\
            & & Z
        \end{tikzcd}\incat{\D}
    \end{equation}
    Since uniqueness of $h$ follows trivially, we only need to construct $h$.
    Let $s:Y\to X$ be a section of $r$ and define $h:=g\circ Gs$.
    Then, the following commutes:
    \begin{equation*}
        \begin{tikzcd}
            FY\arrow[dd,"\alpha_Y"']\arrow[rd,"Fs"']\arrow[rr,equal] & & FY\arrow[dd,"f"] \\
            & FX\arrow[d,"\alpha_X"]\arrow[ru,"Fr"'] & \\
            GY\arrow[rr,"h"',shift right=3]\arrow[r,"Gs"] & GX\arrow[r,"g"] & Z
        \end{tikzcd}\incat{\D}
    \end{equation*}
    Since $\alpha_X$ is an epimorphism, $h\circ Gr=g$ holds.
    Thus, the diagram \cref{eq:canonical_morphism_h} commutes.
\end{proof}

\begin{lemma}\label{lem:equifier_closed_under}
    Let $F,G:\C\to\D$ be functors and let $\alpha,\beta:F\Rightarrow G$ be natural transformations.
    Let $\E\subseteq\C$ be the equifier of $\alpha$ and $\beta$, i.e., 
    $\E\subseteq\C$ is a full subcategory of $\C$ defined by $\E:=\{ X\in\C \mid \alpha_X=\beta_X \}$.
    Then, the following hold:
    \begin{enumerate}
        \item\label{lem:equifier_closed_under-1}
        Let $m:X\to Y$ be a morphism in $\C$ such that $Gm$ is a monomorphism in $\D$.
        Then, $Y\in\E$ implies $X\in\E$.
        \item\label{lem:equifier_closed_under-2}
        Let $p:X\to Y$ be a morphism in $\C$ such that $Fp$ is an epimorphism in $\D$.
        Then, $X\in\E$ implies $Y\in\E$.
    \end{enumerate}
\end{lemma}
\begin{proof}
    Let $m:X\to Y$ be a morphism in $\C$ such that $Gm$ is a monomorphism.
    Consider the following diagram:
    \begin{equation*}
        \begin{tikzcd}
            FX\arrow[d,"\alpha_X"',shift right=1]\arrow[d,"\beta_X",shift left=1]\arrow[r,"Fm"] & FY\arrow[d,"\alpha_{Y}"',shift right=1]\arrow[d,"\beta_{Y}",shift left=1] \\
            GX\arrow[r,"Gm"] & GY
        \end{tikzcd}\incat{\D}
    \end{equation*}
    If $\alpha_Y=\beta_Y$ holds, then $(Gm)\circ\alpha_X=(Gm)\circ\beta_X$ holds, hence $\alpha_X=\beta_X$.
    This proves \ref{lem:equifier_closed_under-1}.
    \ref{lem:equifier_closed_under-2} is dual to \ref{lem:equifier_closed_under-1}.
\end{proof}

\begin{theorem}\label{thm:from_monad_to_alg}
    Let $T$ be a finitary monad on $\A:=\bS\-\PMod$.
    Then, there exist a set $E$ of $\bS$-relative judgments and an isomorphism $\A^T\cong\Alg(\Omega_T,E)$ which commutes with forgetful functors.
    \begin{equation*}
        \begin{tikzcd}
            \A^T\arrow[rd,"U^T"']\arrow[r,phantom,"\cong"] &[-10pt] \Alg(\Omega_T,E)\arrow[d,"\text{forgetful}"sloped]\arrow[r,phantom,"\subseteq"] &[-10pt] \Alg\Omega_T\arrow[ld,"U"] \\[15pt]
            & \A &
        \end{tikzcd}
    \end{equation*}
\end{theorem}
\begin{proof}
    By \cref{lem:componentwize_epi_induce_fullyfaithful,lem:alpha_is_dense}, the functor $\Alg\alpha_T:\Alg T\to\Alg H_{\Omega_T}$ is fully faithful and injective on objects and has a replete image.
    Since $\Alg H_{\Omega_T}\cong\Alg\Omega_T$ by \cref{lem:iso_alg_alg}, $\Alg T$ is isomorphic to a replete full subcategory of $\Alg\Omega_T$.
    The Eilenberg-Moore category $\A^T$ is now isomorphic to a replete full subcategory of $\Alg\Omega_T$ since it is a replete full subcategory of $\Alg T$.
    Note that the inclusion $\A^T\hookrightarrow\Alg\Omega_T$ commutes with forgetful functors, that is, the following commutes:
    \begin{equation*}
        \begin{tikzcd}
            \A^T\arrow[rrd,"U^T"']\arrow[r,phantom,"\subseteq"] &[-10pt] \Alg T\arrow[rd,"\text{forgetful}"sloped]\arrow[r,"\Alg\alpha_T",hook] &[-10pt] \Alg H_{\Omega_T}\arrow[d,"\text{forgetful}"sloped]\arrow[r,phantom,"\cong"] &[-10pt] \Alg\Omega_T\arrow[ld,"U"] \\[25pt]
            & & \A &
        \end{tikzcd}
    \end{equation*}
    To conclude that there exists a set $E$ of $\bS$-relative judgments such that $\A^T\cong\Alg(\Omega_T,E)$, we use the Birkhoff's theorem for relative algebraic theories (\cref{thm:birkhoff_for_rat}).
    That is, it suffices to show that both $\A^T\subseteq\Alg T$ and $\Alg\alpha_T:\Alg T\hookrightarrow\Alg H_{\Omega_T}$ are closed under products, $\Sigma$-closed subobjects, $U$-retracts and filtered colimits.

    Since $\A^T$, $\Alg T$ and $\Alg H_{\Omega_T}$ are finitary monadic over $\A$ by \cref{lem:finitary_endofunctor_monadic_adjoint}, their forgetful functors create products and filtered colimits.
    Therefore, both $\A^T\subseteq\Alg T$ and $\Alg\alpha_T:\Alg T\hookrightarrow\Alg H_{\Omega_T}$ are closed under products and filtered colimits.

    To show that $\Alg\alpha_T:\Alg T\hookrightarrow\Alg H_{\Omega_T}$ is closed under $\Sigma$-closed subobjects, let $(Y,y)$ be an object in $\Alg T$ and let $m:(X,x)\to (Y,y\circ\alpha_{T,Y})$ be a morphism in $\Alg H_{\Omega_T}$ such that $m:X\to Y$ is a $\Sigma$-closed monomorphism in $\A$.
    Then, the following commutes:
    \begin{equation*}
        \begin{tikzcd}
            H_{\Omega_T}X\arrow[dd,"\alpha_{T,X}"']\arrow[rd,"H_{\Omega_T}(m)"description]\arrow[rr,"x"] & & X\arrow[dd,"m"] \\
            & H_{\Omega_T}Y\arrow[d,"\alpha_{T,Y}"] & \\
            TX\arrow[r,"Tm"] & TY\arrow[r,"y"] & Y
        \end{tikzcd}\incat{\A}
    \end{equation*}
    Since $\alpha_{T,X}$ is $\Sigma$-dense by \cref{lem:alpha_is_dense} and $m$ is $\Sigma$-closed, the above rectangle has a unique diagonal filler by \cref{thm:dense_closedmono_factorization}.
    Thus, $(X,x)$ lies in the image of $\Alg\alpha_T$, hence $\Alg\alpha_T:\Alg T\hookrightarrow\Alg H_{\Omega_T}$ is closed under $\Sigma$-closed subobjects.

    To show that $\Alg\alpha_T:\Alg T\hookrightarrow\Alg H_{\Omega_T}$ is closed under $U$-retracts, let $(X,x)$ be an object in $\Alg T$ and let $r:(X,x\circ\alpha_{T,X})\to (Y,y)$ be a morphism in $\Alg H_{\Omega_T}$ such that $r:X\to Y$ is a retraction in $\A$.
    By \cref{lem:nat_square_pushout}, the naturality square of $\alpha_T$ for $r$ forms a pushout.
    Thus, we have the following canonical morphism $y':TY\to Y$:
    \begin{equation*}
        \begin{tikzcd}
            H_{\Omega_T}X\arrow[rd,pos=1.00,phantom,"\ulcorner"]\arrow[d,"\alpha_{T,X}"']\arrow[r,"H_{\Omega_T}(r)"] & H_{\Omega_T}Y\arrow[d,"\alpha_{T,Y}"'description]\arrow[rdd,"y"] & \\
            TX\arrow[d,"x"']\arrow[r,"Tr"'] & TY\arrow[rd,dashed,"y'"'pos=0.3] & \\
            X\arrow[rr,"r"'] & & Y
        \end{tikzcd}\incat{\A}
    \end{equation*}
    Therefore, $(Y,y)$ lies in the image of $\Alg\alpha_T$, hence $\Alg\alpha_T:\Alg T\hookrightarrow\Alg H_{\Omega_T}$ is closed under $U$-retracts.

    It remains to show that $\A^T\subseteq\Alg T$ is closed under $\Sigma$-closed subobjects and $U$-retracts.
    Let $V:\Alg T\to\A$ be the forgetful functor and let $\xi:TV\Rightarrow V$ be the natural transformation defined by $\xi_{(X,x)}:TX\arr{x}X$.
    Let $\eta$ and $\mu$ denote the unit and multiplication of the monad $T$.
    The Eilenberg-Moore category $\A^T$ is now the double equifier of a pair $\id:V\Rightarrow V$ and $\xi\circ\eta V:V\Rightarrow V$ and a pair $\xi\circ T\xi:TTV\Rightarrow V$ and $\xi\circ\mu V:TTV\Rightarrow V$, i.e., $\A^T$ is the full subcategory of $\Alg T$ defined by $\A^T:=\{ (X,x)\in\Alg T \mid \id_{(X,x)}=(\xi\circ\eta V)_{(X,x)}\text{~and~} (\xi\circ T\xi)_{(X,x)}=(\xi\circ\mu V)_{(X,x)} \}$.
    Note that every $\Sigma$-closed monomorphism in $\Alg T$ is transferred to a $\Sigma$-closed monomorphism by $V$ and that every $U$-retraction in $\Alg T$ is transferred to a retraction by both $TTV$ and $V$.
    Thus, \cref{lem:equifier_closed_under} now shows that $\A^T\subseteq\Alg T$ is closed under $\Sigma$-closed subobjects and $U$-retracts.
    This completes the proof.
\end{proof}

\begin{notation}
    Given a category $\C$, let denote by $\Mndf(\C)$ the category of finitary monads on $\C$ and monad morphisms in the sense of \cite[3 Section 6]{barr2005ttt}.
\end{notation}

Combining \cref{thm:from_alg_to_monad,thm:from_monad_to_alg}, we obtain the characterization theorem for our relative algebraic theories:

\begin{theorem}\label{thm:equiv_between_monad_and_rat}
    Let $\bS$ be a partial Horn theory over $\Sigma$.
    Then, there is an equivalence $\Th^\bS\simeq\Mndf(\bS\-\PMod)$ which makes the following commute up to isomorphism:
    \begin{equation*}
        \begin{tikzcd}
            \Th^\bS\arrow[rd,"\Alg"'name=Alg]\arrow[rr,"\simeq"] &[-20pt] &[-20pt] \Mndf(\bS\-\PMod)\arrow[ld,"\EM"] \\
            & (\CAT/\bS\-\PMod)^\op &
            \arrow[from=Alg,to=1-3,phantom,"\cong"sloped]
        \end{tikzcd}
    \end{equation*}
    Here, $\EM$ is the functor which assigns the Eilenberg-Moore category to each finitary monad.
\end{theorem}
\begin{proof}
    We now define a functor $K:\Th^\bS\to\Mndf(\bS\-\PMod)$.
    For each $\bS$-relative algebraic theory $(\Omega,E)$, let $K(\Omega,E):=\mnd{\Omega}{E}$ be the monad induced by the following adjunction:
    \begin{equation*}
    \begin{tikzcd}[column sep=large, row sep=large]
        \bS\-\PMod\arrow[r,"\bF",shift left=7pt]\arrow[r,"\perp"pos=0.5,phantom] &[10pt]\Alg(\Omega,E)\arrow[l,"U",shift left=7pt]
    \end{tikzcd}
    \end{equation*}
    By \cref{thm:from_alg_to_monad}, we have $\Alg(\Omega,E)\cong\EM(\mnd{\Omega}{E})$.
    Since the functor $\EM$ is fully faithful by \cite[3 Theorem 6.3]{barr2005ttt}, for each morphism $[\rho]:(\Omega,E)\to (\Omega',E')$ in $\Th^\bS$, we have a unique monad morphism $K[\rho]:\mnd{\Omega}{E}\to\mnd{\Omega'}{E'}$ which makes the following commute:
    \begin{equation*}
        \begin{tikzcd}
            \EM(\mnd{\Omega}{E})\arrow[d,phantom,"\cong"sloped] & & \EM(\mnd{\Omega'}{E'})\arrow[ll,"\EM(K\lbrack\rho\rbrack)"']\arrow[d,phantom,"\cong"sloped] \\[-10pt]
            \Alg(\Omega,E)\arrow[rd,"U"'] & & \Alg(\Omega',E')\arrow[ll,"\Alg\rho"']\arrow[ld,"U'"] \\
            & \bS\-\PMod &
        \end{tikzcd}
    \end{equation*}
    This yields a desired functor $K:\Th^\bS\to\Mndf(\bS\-\PMod)$.
    Then, $K$ is fully faithful by \cref{thm:alg_fully_faithful} and essentially surjective by \cref{thm:from_monad_to_alg}, hence $K$ is an equivalence.
\end{proof}

In particular, we get a syntactic presentation of finitary monadic categories over locally finitely presentable categories:

\begin{corollary}
    Let $\bS$ be a partial Horn theory.
    For each category $\C$, the following are equivalent:
    \begin{enumerate}
        \item $\C$ is finitary monadic (resp. strictly finitary monadic) over $\bS\-\PMod$.
        \item $\C$ is equivalent (resp. isomorphic) to $\Alg(\Omega,E)$ for some $\bS$-relative algebraic theory $(\Omega,E)$.
    \end{enumerate}
\end{corollary}

\begin{corollary}\label{cor:S-rat_is_independent_of_S}
    The concept of relative algebraic theories is independent of the choice of $\bS$, i.e., 
    if there exists an equivalence $\bS\-\PMod\simeq\bT\-\PMod$ with two partial Horn theories $\bS$ and $\bT$, then the following classes of categories coincide (up to equivalence):
    \begin{enumerate}
        \item Categories of models of $\bS$-relative algebraic theories.
        \item Categories of models of $\bT$-relative algebraic theories.
    \end{enumerate}
\end{corollary}

Thus, given a locally finitely presentable category $\A$, we may call an $\bS$-relative algebraic theory an \emph{$\A$-relative algebraic theory} as long as categories of models are concerned, taking an arbitrary partial Horn theory $\bS$ such that $\A\simeq\bS\-\PMod$.

\subsection{A characterization of total algebras}
In the previous subsections, we observed the connection of relative algebraic theories and finitary monads on locally finitely presentable categories.
In this subsection, we will describe the case of \emph{sifted-colimit-preserving} monads.
For more detail of \emph{sifted colimits}, we refer the reader to \cite{adamek2010algebraic}.

\vspace{5em}
\begin{definition}\quad
    \begin{enumerate}
        \item
        A category $\C$ is \emph{sifted} if every colimit indexed by $\C$ commutes with finite products in $\Set$.
        \item
        A \emph{sifted colimit} is a colimit indexed by a small sifted category.
        \item
        A \emph{reflexive pair} is a parallel pair
        $
        \begin{tikzcd}
            \cdot\arrow[r,shift left=1,"u"]\arrow[r,shift right=1,"v"'] & \cdot
        \end{tikzcd}
        $
        with a common section, i.e., a morphism $s$ such that $us=vs=\id$.
        \item
        A \emph{reflexive coequalizer} is a coequalizer of a reflexive pair.
    \end{enumerate}
\end{definition}

\begin{proposition}
    Every reflexive coequalizer is a sifted colimit.
\end{proposition}

\begin{definition}
    A partial Horn theory $(S,\Sigma,\bT)$ is called an \emph{equational theory} if
    \begin{itemize}
        \item
        $\Sigma$ contains no relation symbol,
        \item
        every function symbol $f$ in $\Sigma$ is total, i.e., the sequent $\top \seq{\tup{x}} f(\tup{x})\defined$ is a PHL-theorem of $\bT$, and
        \item
        $\bT$ consists of equations, i.e., every sequent in $\bT$ has the expression $\top\seq{\tup{x}}\phi$.
    \end{itemize}
    Given an equational theory $\bT$, we will denote the category $\bT\-\PMod$ by $\bT\-\Mod$.
    A category equivalent to $\bT\-\Mod$ for some equational theory $\bT$ is sometimes called a \emph{variety}.
\end{definition}

The class of varieties is known to have various characterizations:
\begin{theorem}\label{thm:varieties}
    The following classes of categories coincide:
    \begin{enumerate}
        \item
        Categories equivalent to the Eilenberg-Moore category of a finitary monad on $\Set^S$ for some set $S$.
        \item
        Categories equivalent to the Eilenberg-Moore category of a sifted-colimit-preserving monad on $\Set^S$ for some set $S$.
        \item
        Varieties, i.e., categories of models of equational theories.
        \item
        Categories of models of finite product theories, which are sometimes called (multi-sorted) Lawvere theories.
        \item
        Locally $\mathbf{FINPR}$-presentable categories. (in the sense of $\bD:=\mathbf{FINPR}$ in \cite{adamek2002classification})
    \end{enumerate}
\end{theorem}
\begin{proof}
    See \cite[Proposition 3.3]{rosicky2012strongly}, \cite[14.28 Proposition, A.40 Theorem]{adamek2010algebraic}, and \cite[Theorem 5.5]{adamek2002classification}.
\end{proof}

\begin{theorem}\label{thm:variety_on_variety}
    Let $\bS$ be an equational theory over an $S$-sorted signature $\Sigma$.
    Then for each category $\C$, the following are equivalent:
    \begin{enumerate}
        \item\label{thm:variety_on_variety-1}
        $\C$ is monadic over $\bS\-\Mod$ and its monad preserves sifted colimits.
        \item\label{thm:variety_on_variety-2}
        $\C$ is equivalent to $\Alg(\Omega,E)$ for some $\bS$-relative algebraic theory $(\Omega,E)$ whose operators are total and whose axioms have the expression $\top\seq{\tup{x}}\phi$.
        \item\label{thm:variety_on_variety-3}
        $\C$ is equivalent to $\bT\-\Mod$ for some $S$-sorted equational theory $\bT$ and there exists a theory morphism $\rho:\bS\to\bT$ which is identity on sorts.
    \end{enumerate}
\end{theorem}
\begin{proof}
    It is easily seen that \ref{thm:variety_on_variety-2} and \ref{thm:variety_on_variety-3} are equivalent.
    
    {[\ref{thm:variety_on_variety-1}$\implies$\ref{thm:variety_on_variety-3}]}
    Let $T$ be a sifted-colimit-preserving monad on $\bS\-\Mod$ and consider the following adjunctions:
    \begin{equation*}
        \begin{tikzcd}
            \Set^S\arrow[r,"F",shift left=7pt]\arrow[r,"\perp"pos=0.5,phantom]
            &[10pt]\bS\-\Mod\arrow[l,"U",shift left=7pt]\arrow[r,"F^T",shift left=7pt]\arrow[r,"\perp"pos=0.5,phantom]
            &[10pt](\bS\-\Mod)^T\arrow[l,shift left=7pt,"U^T"]
        \end{tikzcd}
    \end{equation*}
    Since $U^T$ creates sifted colimits, $U\circ U^T$ is monadic by Beck's crude monadicity theorem.
    Thus, the proof is completed by \cref{thm:varieties}.
    
    {[\ref{thm:variety_on_variety-3}$\implies$\ref{thm:variety_on_variety-1}]}
    Consider the following adjunctions:
    \begin{equation*}
        \begin{tikzcd}
            \bS\-\Mod & & \bT\-\Mod \\[20pt]
            & \Set^S &
            \arrow[from=1-1,to=1-3,shorten=6,shift left=2,"F^\rho"]
            \arrow[from=1-3,to=1-1,shorten=6,shift left=2,"U^\rho"]
            \arrow[from=2-2,to=1-1,shorten=6,shift left=2,"F_\bS"]
            \arrow[from=1-1,to=2-2,shorten=6,shift left=2,"U_\bS"]
            \arrow[from=2-2,to=1-3,shorten=6,shift left=2,"F_\bT"]
            \arrow[from=1-3,to=2-2,shorten=6,shift left=2,"U_\bT"]
            \arrow[from=1-1,to=1-3,phantom,"\perp"]
            \arrow[from=1-1,to=2-2,phantom,"\top"sloped]
            \arrow[from=2-2,to=1-3,phantom,"\perp"sloped]
        \end{tikzcd}
    \end{equation*}
    Here $U_\bS, U_\bT$ are forgetful functors and $U_\bT=U_\bS\circ U^\rho$ holds.
    By \cref{thm:varieties}, $U_\bS, U_\bT$ are monadic and create absolute coequalizers, hence $U^\rho$ also creates absolute coequalizers.
    Thus Beck's monadicity theorem asserts that $U^\rho$ is also monadic.
    Since both $U_\bS$ and $U_\bT$ create sifted colimits, so does $U^\rho$, which completes the proof.
\end{proof}

The above theorem asserts that sifted-colimit-preserving monads on varieties characterize total algebras whose axioms are Horn formulas.
However, this characterization fails on locally finitely presentable categories in general:
\begin{example}
    Let us define a $\Pos$-relative algebraic theory $\Omega$ with no axiom as $\Omega:=\{\omega\}$ and $\ar(\omega):=(x.\top)$.
    Then an $\Omega$-algebra $\bX$ is a poset $X$ with endomorphism $\intpn{\omega}{\bX}:X\to X$ in $\Set$, which does not necessarily preserve orders.
    We now show that the forgetful functor $U:\Alg\Omega\to\Pos$ does not preserve sifted colimits.
    Consider the $\Omega$-algebras $\bX$ and $\1$ given by
    \begin{equation*}
        \bX:=\left(
        \begin{tikzcd}
            c\arrow[loop left,mapsto] & \\
            b\arrow[u,phantom,"<"sloped]\arrow[r,mapsto] & d\arrow[loop right,mapsto] \\
            a\arrow[u,phantom,"<"sloped]\arrow[loop left,mapsto] &
        \end{tikzcd}
        \right)
        \quad\text{and}\quad
        \1:=\left(
        \begin{tikzcd}
            *\arrow[loop right,mapsto]
        \end{tikzcd}
        \right).
    \end{equation*}
    In the above, the symbol $<$ expresses the partial order of $X$ and arrows express endomorphisms.
    We see at once that the following forms a reflexive coequalizer in $\Alg\Omega$:
    \begin{equation*}
        \begin{tikzcd}[column sep=large,row sep=large]
            \1+\bX\arrow[r,shift left=2,"{(\const{c},\id)}"]\arrow[r,shift right=2,"{(\const{a},\id)}"'] & \bX\arrow[r,"!"] & \1
        \end{tikzcd}\incat{\Alg\Omega}.
    \end{equation*}
    However, the above is not a coequalizer in $\Pos$.
    Indeed, the correct coequalizer in $\Pos$ has two different elements.
\end{example}

\printbibliography

@book{adamek1994locally,
    AUTHOR = {Ad\'{a}mek, Ji\v{r}\'{\i} and Rosick\'{y}, Ji\v{r}\'{\i}},
     TITLE = {Locally Presentable and Accessible Categories},
    SERIES = {London Mathematical Society Lecture Note Series},
    VOLUME = {189},
 PUBLISHER = {Cambridge University Press, Cambridge},
      YEAR = {1994},
     %PAGES = {xiv+316},
      %ISBN = {0-521-42261-2},
   %MRCLASS = {18Axx (18-02)},
  %MRNUMBER = {1294136},
%MRREVIEWER = {J. R. Isbell},
       %DOI = {10.1017/CBO9780511600579},
       %URL = {https://DOI.org/10.1017/CBO9780511600579},
}

@book{adamek2010algebraic,
    AUTHOR = {Ad\'{a}mek, Ji\v{r}\'{\i} and Rosick\'{y}, Ji\v{r}\'{\i} and Vitale, E. M.},
     TITLE = {Algebraic Theories},
    SERIES = {Cambridge Tracts in Mathematics},
    VOLUME = {184},
      %NOTE = {A categorical introduction to general algebra, With a foreword by F. W. Lawvere},
 PUBLISHER = {Cambridge University Press, Cambridge},
      YEAR = {2011},
     %PAGES = {xviii+249},
      %ISBN = {978-0-521-11922-1},
   %MRCLASS = {18-02 (18C05 18C10 18C15)},
  %MRNUMBER = {2757312},
%MRREVIEWER = {Walter Tholen},
}

@book{borceux1994handbook1,
    AUTHOR = {Borceux, Francis},
     TITLE = {Handbook of Categorical Algebra. 1},
    SERIES = {Encyclopedia of Mathematics and its Applications},
    VOLUME = {50},
      NOTE = {Basic category theory},
 PUBLISHER = {Cambridge University Press, Cambridge},
      YEAR = {1994},
     %PAGES = {xvi+345},
      %ISBN = {0-521-44178-1},
   %MRCLASS = {18-02 (18Axx)},
  %MRNUMBER = {1291599},
%MRREVIEWER = {Martin Hyland},
}

@book{johnstone2002sketches,
    AUTHOR = {Johnstone, Peter T.},
     TITLE = {Sketches of an Elephant: A Topos Theory Compendium. {V}ol. 2},
    SERIES = {Oxford Logic Guides},
    VOLUME = {44},
 PUBLISHER = {The Clarendon Press, Oxford University Press, Oxford},
      YEAR = {2002},
     %PAGES = {i--xxii, 469--1089 and I1--I71},
      %ISBN = {0-19-851598-7},
   %MRCLASS = {18B25 (03B15 03G30 18-02 54A05)},
  %MRNUMBER = {2063092},
%MRREVIEWER = {Colin McLarty},
}

@article{palmgren2007partial,
    AUTHOR = {Palmgren, E. and Vickers, S. J.},
     TITLE = {Partial Horn logic and cartesian categories},
   JOURNAL = {Ann. Pure Appl. Logic},
  %FJOURNAL = {Annals of Pure and Applied Logic},
    VOLUME = {145},
      YEAR = {2007},
    NUMBER = {3},
     PAGES = {314--353},
      %ISSN = {0168-0072},
   %MRCLASS = {03G30 (03C05 18C10)},
  %MRNUMBER = {2286417},
%MRREVIEWER = {Anna Labella},
       %DOI = {10.1016/j.apal.2006.10.001},
       %URL = {https://DOI.org/10.1016/j.apal.2006.10.001},
}

@incollection{ford2021monads,
    AUTHOR = {Ford, Chase and Milius, Stefan and Schr\"{o}der, Lutz},
     TITLE = {Monads on categories of relational structures},
 BOOKTITLE = {9th {C}onference on {A}lgebra and {C}oalgebra in {C}omputer
              {S}cience},
    SERIES = {LIPIcs. Leibniz Int. Proc. Inform.},
    VOLUME = {211},
     PAGES = {Art. No. 14, 17},
 PUBLISHER = {Schloss Dagstuhl. Leibniz-Zent. Inform., Wadern},
      YEAR = {2021},
   %MRCLASS = {03C05 (08A70 18C15)},
  %MRNUMBER = {4390209},
}

@article{adamek2021finitary,
    AUTHOR = {Ad\'{a}mek, Ji\v{r}\'{\i} and Ford, Chase and Milius, Stefan and Schr\"{o}der,
              Lutz},
     TITLE = {Finitary monads on the category of posets},
   JOURNAL = {Math. Structures Comput. Sci.},
%   FJOURNAL = {Mathematical Structures in Computer Science. A Journal in the
%              Applications of Categorical, Algebraic and Geometric Methods
%              in Computer Science},
    VOLUME = {31},
      YEAR = {2021},
    NUMBER = {7},
     PAGES = {799--821},
      %ISSN = {0960-1295},
   %MRCLASS = {18C15 (08Bxx)},
  %MRNUMBER = {4386831},
%MRREVIEWER = {Mahdieh Haddadi},
       %DOI = {10.1017/S0960129521000360},
       %URL = {https://DOI.org/10.1017/S0960129521000360},
}

@article{adamek2009orthogonal,
    AUTHOR = {Ad\'{a}mek, Ji\v{r}\'{\i} and H\'{e}bert, Michel and Sousa, Lurdes},
     TITLE = {The orthogonal subcategory problem and the small object
              argument},
   JOURNAL = {Appl. Categ. Structures},
%   FJOURNAL = {Applied Categorical Structures. A Journal Devoted to
%              Applications of Categorical Methods in Algebra, Analysis,
%              Order, Topology and Computer Science},
    VOLUME = {17},
      YEAR = {2009},
    NUMBER = {3},
     PAGES = {211--246},
      %ISSN = {0927-2852},
   %MRCLASS = {18A40 (18G05 55P99)},
  %MRNUMBER = {2506255},
%MRREVIEWER = {Philippe Gaucher},
       %DOI = {10.1007/s10485-008-9153-4},
       %URL = {https://DOI.org/10.1007/s10485-008-9153-4},
}

@article{baron1969reflectors,
    AUTHOR = {Baron, S.},
     TITLE = {Reflectors as compositions of epi-reflectors},
   JOURNAL = {Trans. Amer. Math. Soc.},
  %FJOURNAL = {Transactions of the American Mathematical Society},
    VOLUME = {136},
      YEAR = {1969},
     PAGES = {499--508},
      %ISSN = {0002-9947},
   %MRCLASS = {18.10},
  %MRNUMBER = {236237},
%MRREVIEWER = {H.-B. Brinkmann},
       %DOI = {10.2307/1994729},
       %URL = {https://DOI.org/10.2307/1994729},
}

@book{barr2005ttt,
    AUTHOR = {Barr, Michael and Wells, Charles},
     TITLE = {Toposes, Triples and Theories},
      NOTE = {Corrected reprint of the 1985 original [MR0771116]},
   JOURNAL = {Repr. Theory Appl. Categ.},
    NUMBER = {12},
      YEAR = {2005},
}

@incollection{linton1969outline,
    AUTHOR = {Linton, F. E. J.},
     TITLE = {An outline of functorial semantics},
 BOOKTITLE = {Sem. on {T}riples and {C}ategorical {H}omology {T}heory
              ({ETH}, {Z}\"{u}rich, 1966/67)},
     PAGES = {7--52},
 PUBLISHER = {Springer, Berlin},
      YEAR = {1969},
   %MRCLASS = {18.10},
  %MRNUMBER = {0244340},
%MRREVIEWER = {J. Sonner},
}

@article{hebert2004algebraically,
    AUTHOR = {H\'{e}bert, Michel},
     TITLE = {Algebraically closed and existentially closed substructures in
              categorical context},
   JOURNAL = {Theory Appl. Categ.},
  %FJOURNAL = {Theory and Applications of Categories},
    VOLUME = {12},
      YEAR = {2004},
     PAGES = {No. 9, 270--298},
   %MRCLASS = {18C35 (03C60 18A20)},
  %MRNUMBER = {2056100},
%MRREVIEWER = {J\v{i}r\'{\i} Rosick\'{y}},
}

@incollection{adamek2002classification,
    AUTHOR = {Ad\'{a}mek, Ji\v{r}\'{\i} and Borceux, Francis and Lack, Stephen and Rosick\'{y}, Ji\v{r}\'{\i}},
     TITLE = {A classification of accessible categories},
      NOTE = {Special volume celebrating the 70th birthday of Professor Max
              Kelly},
   JOURNAL = {J. Pure Appl. Algebra},
  %FJOURNAL = {Journal of Pure and Applied Algebra},
    VOLUME = {175},
      YEAR = {2002},
    NUMBER = {1-3},
     PAGES = {7--30},
      %ISSN = {0022-4049},
   %MRCLASS = {18C35 (18A30 18A35 18C30)},
  %MRNUMBER = {1935970},
%MRREVIEWER = {R. H. Street},
       %DOI = {10.1016/S0022-4049(02)00126-3},
       %URL = {https://DOI.org/10.1016/S0022-4049(02)00126-3},
}

@article{kelly1993adjunctions,
    AUTHOR = {Kelly, G. M. and Power, A. J.},
     TITLE = {Adjunctions whose counits are coequalizers, and presentations of finitary enriched monads},
   JOURNAL = {J. Pure Appl. Algebra},
  %FJOURNAL = {Journal of Pure and Applied Algebra},
    VOLUME = {89},
      YEAR = {1993},
    NUMBER = {1-2},
     PAGES = {163--179},
      %ISSN = {0022-4049},
   %MRCLASS = {18C10 (18C15 18D20)},
  %MRNUMBER = {1239558},
%MRREVIEWER = {Walter Tholen},
       %DOI = {10.1016/0022-4049(93)90092-8},
       %URL = {https://DOI.org/10.1016/0022-4049(93)90092-8},
}

@book{caramello2018theories,
    AUTHOR = {Caramello, Olivia},
     TITLE = {Theories, Sites, Toposes},
      NOTE = {Relating and studying mathematical theories through
              topos-theoretic `bridges'},
 PUBLISHER = {Oxford University Press, Oxford},
      YEAR = {2018},
     %PAGES = {xii+368},
      %%ISBN = {978-0-19-875891-4},
   %MRCLASS = {18-02 (03G30 18B25 18C10 18C50 18F20)},
  %MRNUMBER = {3752150},
%MRREVIEWER = {Andrzej Wi\'{s}nicki},
}

@article{rosicky2012strongly,
    AUTHOR = {Kurz, Alexander and Rosick\'{y}, Ji\v{r}\'{\i}},
     TITLE = {Strongly complete logics for coalgebras},
   JOURNAL = {Log. Methods Comput. Sci.},
  %FJOURNAL = {Logical Methods in Computer Science},
    VOLUME = {8},
      YEAR = {2012},
    NUMBER = {3},
     PAGES = {3:14, 32},
   %MRCLASS = {68Q65 (03B45 08C20 18A30 18C10)},
  %MRNUMBER = {2981933},
%MRREVIEWER = {Lurdes Sousa},
       %DOI = {10.2168/LMCS-8(3:14)2012},
       %URL = {https://DOI.org/10.2168/LMCS-8(3:14)2012},
}

@book{makkai1989accessible,
    AUTHOR = {Makkai, Michael and Par\'{e}, Robert},
     TITLE = {Accessible Categories: The Foundations of Categorical Model Theory},
    SERIES = {Contemporary Mathematics},
    VOLUME = {104},
 PUBLISHER = {American Mathematical Society, Providence, RI},
      YEAR = {1989},
     %PAGES = {viii+176},
      %ISBN = {0-8218-5111-X},
   %MRCLASS = {03G30 (03C75 03C95 18B25 18C10 18D05 18F10)},
  %MRNUMBER = {1031717},
%MRREVIEWER = {Horst Reichel},
       %DOI = {10.1090/conm/104},
       %URL = {https://DOI.org/10.1090/conm/104},
}

@misc{bidlingmaier2018categories,
  title = {Categories with algebraic structure as models of partial Horn logic theories},
  author = {Bidlingmaier, Martin E.},
  year = {2018},
  url = {https://www.mbid.me/masters-thesis.pdf}
}

@misc{burmeister2002lecture,
  title={Lecture notes on universal algebra. Many-sorted partial algebras},
  author={Burmeister, Peter},
  year={2002},
  %url={https://citeseerx.ist.psu.edu/document?repid=rep1&type=pdf&%DOI=0caea53a97c89ea110f78cdb72d11585e4d32f55}
}

@inproceedings{birkhoff1935structure,
  title={On the structure of abstract algebras},
  author={Birkhoff, Garrett},
  booktitle={Mathematical proceedings of the Cambridge philosophical society},
  volume={31},
  number={4},
  pages={433--454},
  year={1935},
  organization={Cambridge University Press}
}

@book{maclane1998working,
    AUTHOR = {Mac Lane, Saunders},
     TITLE = {Categories for the Working Mathematician},
    SERIES = {Graduate Texts in Mathematics},
    VOLUME = {5},
   EDITION = {Second},
 PUBLISHER = {Springer-Verlag, New York},
      YEAR = {1998},
     %PAGES = {xii+314},
      %ISBN = {0-387-98403-8},
   %MRCLASS = {18-02},
  %MRNUMBER = {1712872},
}

@article{berg2012noncomm,
    AUTHOR = {van den Berg, Benno and Heunen, Chris},
     TITLE = {Noncommutativity as a colimit},
   JOURNAL = {Appl. Categ. Structures},
  %FJOURNAL = {Applied Categorical Structures. A Journal Devoted to Applications of Categorical Methods in Algebra, Analysis, Order, Topology and Computer Science},
    VOLUME = {20},
      YEAR = {2012},
    NUMBER = {4},
     PAGES = {393--414},
      %ISSN = {0927-2852,1572-9095},
   %MRCLASS = {16B50 (06E15 46L05 46L85)},
  %MRNUMBER = {2943637},
%MRREVIEWER = {Manuel\ Cort\'{e}s Izurdiaga},
       %DOI = {10.1007/s10485-011-9246-3},
       %URL = {https://doi.org/10.1007/s10485-011-9246-3},
}

@article{rosicky2021metric,
    AUTHOR = {Rosick\'{y}, Ji\v{r}\'{\i}},
     TITLE = {Metric monads},
   JOURNAL = {Math. Structures Comput. Sci.},
  %FJOURNAL = {Mathematical Structures in Computer Science. A Journal in the Applications of Categorical, Algebraic and Geometric Methods in Computer Science},
    VOLUME = {31},
      YEAR = {2021},
    NUMBER = {5},
     PAGES = {535--552},
      %ISSN = {0960-1295,1469-8072},
   %MRCLASS = {08B05 (18C15 18D20 54E35)},
  %MRNUMBER = {4372551},
%MRREVIEWER = {Sergey\ A.\ Solovyov},
       %DOI = {10.1017/s0960129521000220},
       %URL = {https://doi.org/10.1017/s0960129521000220},
}
\end{document}